\documentclass[3p]{elsarticle}

\usepackage{marginnote}

\usepackage[T1]{fontenc}
\usepackage{amsmath}
\usepackage{amssymb}
\usepackage{amsthm}
\usepackage{color}
\usepackage{graphicx}
\usepackage{multirow}
\usepackage{algpseudocode} 
\usepackage{algorithm} 
\usepackage[]{subfig}
\usepackage{accents}
\renewcommand{\d}{\mathrm{d}}
\newcommand{\vect}[1]{\boldsymbol{#1}}

\newtheorem{ass}{Assumption}
\newtheorem{thm}{Theorem}
\newtheorem{lemma}{Lemma}

\newtheorem{rmk}{Remark}

\numberwithin{equation}{section}

\renewcommand{\d}{\mathrm{d}}

\def\GS{\color{black}} 
\def\B{\color{black}}

\newcommand{\ds}{\ d s}

\renewcommand{\S}{\mathcal{S}}

\newcommand{\matb}[1]{\mathbf{#1}}

 \newcommand{\trunc}[1]{\mathsf{#1}}

\def\R1{\color{black}}

\newcommand{\numptc}{\mathcal{N}_{ptc}}
\newcommand{\numsub}{\mathcal{N}_{sub}}

\begin{document}

\begin{frontmatter}

\title {A low-rank solver for conforming multipatch Isogeometric Analysis}

 \author[mat,imati]{Monica Montardini\corref{cor1}} \ead{monica.montardini@unipv.it}
\author[mat,imati]{Giancarlo Sangalli} \ead{giancarlo.sangalli@unipv.it}
\author[mat,imati]{Mattia Tani} \ead{mattia.tani@unipv.it}

\address[mat]{Dipartimento di Matematica ``F. Casorati", Universit\`{a} di Pavia, Via A. Ferrata, 5, 27100 Pavia, Italy.}

\address[imati]{Istituto di Matematica Applicata e Tecnologie Informatiche, ``E. Magenes" del CNR, Via A. Ferrata, 1, 27100 Pavia, Italy.}

 \cortext[cor1]{Corresponding author}

%\pagestyle{myheadings}
%\markboth{M. Montardini, G. Sangalli and M. Tani}{}

%\title {A low-rank multipatch isogeometric solver \thanks{Version of   \today}}
%\title {Low-rank solution of conforming multipatch discretizations for isogeometric analysis \thanks{Version of   \today}}

%\author{ M. Montardini \thanks{Dipartimento di Matematica, Universit\`a degli Studi di Pavia, via Ferrata 5, Pavia, Italy.} \thanks{Istituto di Matematica Applicata e Tecnologie  Informatiche ``E. Magenes'' del CNR, via Ferrata 5/a, Pavia, Italy.  \vskip 1mm \noindent Emails: 
%{monica.montardini@unipv.it,   giancarlo.sangalli@unipv.it, mattia.tani@unipv.it}} \and G. Sangalli $^{\dag \ddag}$ \and  M. Tani $^{\dag \ddag}$ }

%\date{\documentdate}
%
%\maketitle

\begin{abstract} { 
In this paper, we propose an innovative isogeometric low-rank solver
for the linear elasticity model problem, specifically designed to
allow multipatch domains. Our approach splits the domain into
subdomains, each formed by the union of neighboring patches. Within
each subdomain, we employ Tucker low-rank matrices and vectors to
approximate the system matrices and right-hand side vectors,
respectively. This enables the construction of local approximate fast
solvers. These local solvers are then combined into an overlapping
Schwarz preconditioner, which is utilized in a truncated
preconditioned conjugate gradient method. Numerical experiments
demonstrate the significant memory storage benefits and a  uniformly bounded number of iterations with respect to both mesh size and spline degree.}
\end{abstract}


\end{frontmatter}
 
 \section{Introduction}

 Isogeometric Analysis (IgA) has seen substantial advancements since the publication of the seminal work \cite{Hughes2005}. Among its many advantages, IgA stands out as an effective and versatile high-order method for approximating partial differential equations (PDEs). Indeed, IgA benefits from
 the approximation properties of splines, which, especially in the
 case of maximal regularity (splines of polynomial degree $p$ and
 continuity $C^{p-1}$) exhibits approximation properties superior to
 those of classical $C^0$ or discontinuous piecewise polynomials  (see
 for example \cite{Bressan2019a} and  \cite{Evans2009}).
 
As with all high-order numerical methods, the development of efficient
solvers is crucial, and this has been an active area of research
within the IgA community. Among the various approaches proposed,
low-rank compression techniques have recently garnered
attention. These methods take advantage of the tensorial structure
inherent in isogeometric spaces at the patch level. See, for example, 
\cite{Mantzaflaris2017,Hofreither2018,Juttler2017,Pan2019}, where the authors use a low-rank
representation of  the isogeometric  Galerkin matrices.  Low-rank
tensor methods have been also exploited in the
solution of IgA linear systems in
\cite{Georgieva2019,Montardini2023}, based on the
Tucker format, and in \cite{Bunger2020}, that uses the  tensor-trains
approximation of the unknown. {\GS
 In all these papers, a single tensorial patch is considered. This is a limitation that we aim to overcome.

%   The extension to multipatch geometries  is not trivial
% since the multipatch setting is not globally tensorial.  In
% our work, we use a preconditioned Krylov solver
%   with an inexact additive overlapping Schwarz preconditioner, that uses  the  low-rank
%   preconditioner of \cite{Montardini2023}  on suitable tensorial subdomains. The overlapping Schwarz
%   preconditioner is block diagonal for the global linear system where the unknown degrees of freedom of
%  are obtained from the disjoint union of the degrees of freedom of the subdomains: this is the system that we solve by a
%  suitable conjugate gradient method.
In this work, we indeed consider a multipatch domain. This extension
is non-trivial since the multipatch setting is generally not globally
tensorial, even when the individual patches are, as we assume. The use
of classical domain decomposition techniques allows leveraging
tensoriality at the level of local problems but not at the global
level, see our previous work \cite{Bosy2020}.} To circumvent this issue, in the present work, Instead of representing the unknown in the  usual multipatch basis (obtained by a continuous gluing of the individual
  patch bases)  we construct  appropriate tensorial subdomains by
  merging neighboring patches, and  represent  the global unknown as
  the disjoint union of the unknowns of the subdomains. This
  representation is not unique, leading to a singular
  linear system. However, the advantage is that,  on each subdomain,
  we can use a solver that leverages tensorial properties, and
  construct then a block diagonal preconditioner for the global
  system. This is  equivalent, in the full-rank case, to an
  overlapping Schwarz preconditioner for the system written in the
  standard basis, see \cite{BMS-arXiv}.
An additional benefit is that the global
unknown admits a low-rank
approximation on each  subdomain. This allows us to  use the low-rank
preconditioner from \cite{Montardini2023} on each subdomain and, a truncated preconditioned conjugate gradient solver to
solve our global  singular linear system.

It is known that Krylov solvers perform well with
singular systems, provided that the right-hand side lies in the range
of the system matrix, see
e.g. \cite{Kaasschieter1988,Ipsen1998,Reichel2005}. However, in our
approach, low-rank compression is performed at each step, meaning that the iterates are not computed
exactly: in this context, convergence is less understood. %despite the fact that numerical errors always occur due to round-off, see \cite{Gutknecht2000, Sleijpen2001}. 
We also observe that truncated Krylov methods have only been
introduced recently, and their theory is still under development, even
for non-singular systems, see \cite{Kressner2011, Palitta2021,Simoncini2023}. Although there are gaps in the mathematical analysis
and therefore the convergence of the linear solver we use in this work
is not proved, our benchmarks show that the proposed method
works well,  similarly to the single-patch non-singular case considered
in \cite{Montardini2023}. Moreover, it achieves substantial memory compression, reducing storage requirements by up to two orders of magnitude compared to the full-rank case.

{ \GS  In the context of Isogeometric Analysis, only one work, the
  preprint \cite{Bunger2023}, shares the objective of this paper:
  addressing multipatch isogeometric problems by combining low-rank
  techniques with a domain decomposition approach. However, it employs
  a different tensor format (tensor trains) and a non-overlapping
  domain decomposition method in a simplified setting.

  Beyond Isogeometric Analysis, there is significant research activity
  focused on the use of low-rank methods for solving PDEs.  We refer to the recent survey
  \cite{Bachmayr2023} for a throughout discussion on this topic. In particular, several low-rank solvers have been proposed in the literature, including Krylov
  subspace methods with rank truncation \cite{Khoromskij2011,
    Kressner2011, Ballani2013, Billaud2014, Lee2017, Ali2020}, alternating direction approaches
  \cite{Oseledets2012, Dolgov2014, Oseledets2018}, Riemannian
  optimization \cite{Kressner2016, Sutti2024}, and greedy strategies \cite{Ammar2010, Cances2011}. The target problems of these
  works are often parametrized or stochastic PDEs, and the focus is on
  high dimension. We emphasize, however, that in all these references
  PDEs are solved on simple tensor-product domains. To the best of our
  knowledge, the only work proposing a low-rank solver for complex
  domains is \cite{Markeeva2021} (and its follow-up
  \cite{Kornev2024}), but this approach,  different from the one
  presented here,  is restricted to 2-dimensional problems.
}

The paper is organized as follows. In Section \ref{sec:preliminaries}
we present the basics of IgA and of tensor calculus. The core of the
paper is Section~\ref{sec:lr_multipatch}, where we propose the novel
multipatch low-rank strategy. 
 In Section \ref{sec:analysis} we present a  theoretical
  study on the error committed on the system matrix due to the
  low-rank approximation of the geometry.
We present some numerical experiments in Section \ref{sec:numerics} while in the last section we draw some conclusions.

\section{Preliminaries}
\label{sec:preliminaries}
\subsection{B-Splines} 
\label{sec:bsplines}
Let $m$ and $p$ be two positive integers. Then, a knot vector in $[0,1]$ is a set of non-decreasing points $\Xi:=\{\xi_1,\ldots,\xi_{m+p+1}\}$. We consider only open knot vectors, i.e. we set $\xi_1=\dots=\xi_{p+1}=0$ and $\xi_{m+1}=\dots\xi_{m+p+1}=1$. Besides the first and the last ones, knots can be repeated up to multiplicity $p$.

Univariate B-splines are piecewise polynomials $\widehat{b}_{i}:[0,1]\longrightarrow\mathbb{R}$ for $i=1,\dots,m$ of degree $p$ that can be defined from the knot vector $\Xi$ according to Cox-De Boor recursion formulas \cite{DeBoor2001}. We define the mesh-size $h:=\max \{|\xi_{i+1}-\xi_i| \ | \ i=1,\dots,m+p\}$ and we denote the univariate spline space as
$$
\widehat{\mathcal{S}}^p_h:= \text{span} \{\widehat{b}_{i}\ |\ i=1,\dots m\}.
$$
For B-splines properties, we refer to \cite{Cottrell2009}. Multivariate  B-splines are defined by tensor product of univariate B-splines. In this paper we focus on three-dimensional problems. Thus, given $m_d$ and $p_d$ positive integers, we introduce for $d=1,2,3$ the open knot vectors $\Xi_d:=\{\xi_{1,d},\dots,\xi_{m_d+p_d+1,d}\}$, the corresponding B-splines $\widehat{b}_{d,i_d}$, $i_d=1,\ldots,m_d$, mesh-sizes $h_d$ and univariate spline spaces $\widehat{\S}^{p}_{h_{d}}$. The maximum of the mesh-sizes is denoted as $h:=\max\{h_d\ | \ d=1,2,3\}$. Note that for simplicity we are assuming that all univariate B-splines have the same degree.
%, and that the knot vectors are uniform, i.e. the internal knots are equally spaced. 
Multivariate B-splines $\widehat{B}_{ \vect{i}}: [0,1]^3 \rightarrow\mathbb{R} $ are defined as
$$
\widehat{B}_{ \vect{i}}(\underline{\xi}) : =
\widehat{b}_{1,i_1}(\xi_1)   \widehat{b}_{2,i_2}(\xi_2)\widehat{b}_{3,i_3}(\xi_3),
$$
where $\vect{i}:=(i_1,i_2,i_3)$ is a multi-index and  $\underline{\xi} = (\xi_1, \xi_2, \xi_3)$. We define the corresponding spline space as
\begin{equation*}\widehat{\boldsymbol{\S}}^{p}_{h}  := \mathrm{span}\left\{\widehat{B}_{\vect{i}} \ \middle| \text{where } \vect{i}:=(i_1,i_2,i_3)\text{ and }\ i_d = 1,\dots, m_d \text{ for } l=1,2,3 \right\}= {\widehat{\S}^{p}_{h_{1}}\otimes \widehat{\S}^{p}_{h_{2}}\otimes\widehat{\S}^{p}_{h_{3}}}.
\end{equation*}  
%where $\widehat{\S}^{p}_{h_{l}} :=\text{span}\{\widehat{b}_{l,i}\ | \ i=1,\dots,m_l\}$ for $l=1,2,3$.

\subsection{Isogeometric spaces on a multipatch domain}
\label{sec:iso_spaces}
Let $\Omega\subset\mathbb{R}^3$ represent the computational domain, and we assume that it can be written as the union of $\numptc$ non-overlapping closed sets, called patches, i.e. 
%\begin{equation} \label{eq:patches}
$$ \Omega = \bigcup_{i=1}^{\numptc}\Omega^{(i)} $$  
and ${\Omega}^{(i)} \cap\Omega^{(j)}$ has empty interior for $i\neq j$. 
Let $\Gamma:= \left( \bigcup_{i=1}^{\numptc} \partial\Omega^{(i)} \right) \setminus \partial \Omega$ denote the union of patch interfaces.

For a fixed $i \in \left\lbrace 1,\ldots,\numptc \right\rbrace $, we introduce three positive integers
$m_1^{(i)},m_2^{(i)},m_3^{(i)}$ and three open knot vectors $\Xi^{(i)}_d:=\{ \xi_{1,d}^{(i)}, \ldots, \xi^{(i)}_{m_d^{(i)}+p+1,d}\}$, for $d=1,2,3$,where $p$ is another positive integer which does not depend on $i$ and $d$. Let $h^{(i)}_d$ denote the mesh-size of $\Xi^{(i)}_d$ for $d=1,2,3$. Moreover, we introduce $h:=\max\{h_d^{(j)}\ | \ j=1,\dots,\numptc, \; d=1,2,3\}$.
%
% We suppose that the patches are \emph{conforming}, i.e. for all couple of patches such that $\partial\Omega^{(k)}\cap \partial\Omega^{(j)}\neq\emptyset$ and the intersection is not a point, $\Omega^{(k)}$ and $\Omega^{(j)}$ are fully matching (see \cite{Hughes2005}).  
%
The multivariate spline space associated to each patch is denoted as  $\widehat{\boldsymbol{\S}}_{ptc}^{(i)}$. 
We assume that for each patch $\Omega^{(i)}$ there exists a  non-singular parametrization $\mathcal{F}_i\in \left[\widehat{\boldsymbol{\S}}_{ptc}^{(i)}\right]^3$ whose image is $\Omega^{(i)}$, i.e. $\Omega^{(i)}=\mathcal{F}_i([0,1]^3)$ and the Jacobian matrix $J_{\mathcal{F}_i}$ is invertible everywhere. A face of a given patch $\Omega^{(i)}$ is the image through $\mathcal{F}_i$ of a face of the unit cube. Similarly, an edge of $\Omega^{(i)}$ is the image through $\mathcal{F}_i$ of an edge of the unit cube.

%We will also need the spline space with boundary conditions. 
Let $\partial\Omega_D \subseteq \partial \Omega$, which will represent the portion of $\partial\Omega$ endowed with Dirichlet boundary conditions. We need the following assumption.
\begin{ass}
 \label{ass:bc}
We assume that $\partial \Omega_D\cap\partial\Omega^{(i)} $ is either an empty set or an entire face of $\Omega^{(i)}$ for  $i=1,\dots,\numptc$.
\end{ass}

For $i=1,\ldots,\numptc$, we define $\widehat{\mathcal{\boldsymbol\S}}_{ptc,0}^{(i)}\subset \widehat{\mathcal{\boldsymbol\S}}_{ptc}^{(i)}$ as the space generated by the basis functions of $ \widehat{\mathcal{\boldsymbol\S}}_{ptc}^{(i)}$ whose image through $\mathcal{F}^{(i)}$ vanishes on $\partial\Omega_D$.
%By introducing a colexicographical reordering of the basis functions, we can write $\widehat{\mathcal{\boldsymbol\S}}_{ptc,0}^{(i)}:=\textrm{span}\left\{\widehat{B}^{(i)}_{j} \ \middle|  j=1,\dots, n_{ptc}^{(i)}\right\}$, where $n_{ptc}^{(i)}:=\text{dim}\left(\widehat{\mathcal{\boldsymbol\S}}_{ptc,0}^{(i)}\right)$.
Thanks to Assumption \ref{ass:bc},  $\widehat{\mathcal{\boldsymbol\S}}_{ptc,0}^{(i)}$ is a tensor product space, i.e.
$$
\widehat{\mathcal{\boldsymbol\S}}_{ptc,0}^{(i)}:=\widehat{\mathcal{\S}}_{ptc,1}^{(i)}\otimes \widehat{\mathcal{\S}}_{ptc,2}^{(i)}\otimes \widehat{\mathcal{\S}}_{ptc,3}^{(i)},
$$
where $\widehat{\mathcal{\S}}_{ptc,d}^{(i)}:=\text{span}\left\{\widehat{b}^{(i)}_{d,j}\ | \  j=1,\dots,n^{(i)}_{ptc,d}\right\}$ are univariate spline spaces, with $ n^{(i)}_{ptc,d}:=\text{dim}\left(\widehat{\mathcal{\S}}_{ptc,d}^{(i)} \right)$, for $d=1,2,3$. 
%Note that $n_{ptc}^{(i)} = \prod_{d=1}^3 n_{ptc,d}^{(i)}$.

Following the isoparametric concept, we define the local isogeometric space on the $i$-th patch as
$$
V_{ptc}^{(i)}:=\left\lbrace \widehat{v}_h \circ \mathcal{F}_{i}^{-1} \middle| \; \widehat{v}_h \in \widehat{\boldsymbol{\S}}_{ptc,0}^{(i)} \right\rbrace.
$$Note that $\dim\left( V^{(i)}_{ptc}\right) = \prod_{d=1}^3 n_{ptc,d}^{(i)} =: n_{ptc}^{(i)} $. Moreover, the isogeometric space over $\Omega$ is defined as
\begin{equation}
\label{eq:comp_space}
V_h:=\left\{v\in H^1(\Omega)\ \middle| \ v_{|_{\Omega^{(i)}}}\in V_{ptc}^{(i)},\ i=1,\dots,\numptc  \right\}.
\end{equation}

%(DEVO DEFINIRE UNA BASE PER Vh)
  Throughout the paper, we make the following assumption on the meshes.
\begin{ass}
\label{ass:conforming}
We assume that the intersection between two patches, when it is not empty,    is either a full face, a full edge or a vertex.
 Moreover, we require that the meshes are
conforming at the patch interfaces, i.e. for all $i$ and $j$ such that $\partial \Omega^{(i)}\cap\partial\Omega^{(j)}\neq 0$ and the intersection is not a point.   
\end{ass}
Under this assumption, we collect all the basis functions from all local spaces, and then identify those functions whose restriction to the interface $\Gamma$ assumes the same values and it is not identically zero. The resulting set of functions forms a basis for $V_h$. 

\subsection{Model problem}
 \label{sec:model_problem}
 Our model problem is the compressible linear elasticity problem. 
%Let $\Omega\subset\mathbb{R}^3$ be the computational domain, that is represented as the union of $\numptc$ non-overlapping patches $\Omega=\bigcup_{j=1}^{\numptc} \Omega^{(j)}$ (see Section \ref{sec:iso_spaces}).
Let  $\partial \Omega=\partial\Omega_D\cup\partial\Omega_N$ with $\partial\Omega_D\cap \partial\Omega_N=\emptyset$ and where $\partial\Omega_D$ is non-empty and satisfies  Assumption \ref{ass:bc}.   
Let $H^1_D(\Omega):= \left\{v \in H^1(\Omega) \ \middle| \ v=0 \text{ on } \partial\Omega_D \right\}$. For simplicity, we consider only homogeneous Dirichlet boundary conditions, but the non-homogeneous case can be treated similarly. Then, given  $\underline{f} \in  [{L}^2(\Omega)]^3 $ and  $ \ \underline{g} \in   [{L}^2(\partial \Omega_N)]^3$,
we consider the Galerkin problem: find  $\underline{u}\in [H^1_D(\Omega)]^3$ such that for all   $\underline{v}\in [H^1_D(\Omega)]^3$ 
\begin{equation*}   
a(\underline{u},\underline{v})= F( \underline{v}),
\end{equation*}
where we define
\begin{align}
\label{eq:bil}
a(\underline{u},\underline{v}) & :=  2 \mu \int_{\Omega} \varepsilon(\underline{u}) : \varepsilon(\underline{v}) \d{\underline{x}} + \lambda \int_{\Omega} \left( \nabla \cdot \underline{u} \right) \left(\nabla \cdot \underline{v}\right)\; \d{\underline{x}}    , \\
\label{eq:F}
 F( \underline{v}) & :=  \int_{\Omega} \underline{f} \cdot \underline{v} \; \d{\underline{x}} + \int_{\partial \Omega_N} \underline{g}\cdot \underline{v} \ds. 
\end{align}
%$ {\varepsilon}(\underline{v}) := \frac{1}{2} \left(\nabla \underline{v}+ (\nabla \underline{v})^T\right)$ 

Here ${\varepsilon_{i,j}}(\underline{v}) = \frac{1}{2} \left( \frac{\partial u_i}{\partial x_j} + \frac{\partial u_j}{\partial x_i} \right)$ for $i,j=1,2,3$,   while   $\lambda$ and $\mu$ denote the material Lam\'{e} coefficients, that for simplicity we assume to be constant and positive and such that the Poisson ratio $\frac{\lambda}{2(\mu+\lambda)}$ is smaller than   0.5. The corresponding discrete Galerkin problem that we want to solve is the following: find $\underline{u}_h\in [V_h]^3$ such that 
%for all $\underline{v}_h\in [V_h]^3$
\begin{equation} \label{eq:galerkin}
a(\underline{u}_h,\underline{v}_h)= F( \underline{v}_h), \qquad \text{for all } \underline{v}_h\in [V_h]^3, \end{equation}
where the space $V_h$ has been defined in  \eqref{eq:comp_space}. Note that $[V_h]^3 \subseteq [H^1_D(\Omega)]^3$.

\subsection{The Tucker format}
\label{sec:tensors}

Throughout the paper, the entries of a matrix $\matb{B} \in \mathbb{R}^{m_1 \times m_2}$ are denoted with $[\matb{B}]_{i,j}$, $i=1,\ldots,m_1$, $j=1,\ldots,m_2$. Similarly,the entries of a tensor $\mathfrak{B} \in \mathbb{R}^{m_1 \times m_2 \times m_3} $ are denoted with $[\mathfrak{B}]_{i,j,k}$, $i=1,\ldots,m_1$, $j=1,\ldots,m_2$, $k=1,\ldots,m_3$.

Given $\matb{B} \in \mathbb{R}^{m_1 \times m_2}$ and $\matb{C} \in \mathbb{R}^{n_1 \times n_2} $, their Kronecker product is defined as
$$ \matb{B} \otimes \matb{C} = \begin{bmatrix} [\matb{B}]_{1,1} \matb{C} & \ldots & [\matb{B}]_{1,m_2} \matb{C} \\ \vdots & \ddots & \vdots \\ [\matb{B}]_{m_1,1} \matb{C} & \ldots & [\matb{B}]_{m_1,m_2} \matb{C} \end{bmatrix} \in \mathbb{R}^{m_1 n_1 \times m_2 n_2}.$$
%where $[\matb{B}]_{i,j}$ denotes the $ij-$th entry of $\matb{B}$.
The notion of Kronecker product can be generalized to tensors of higher dimensions. In particular, given $\mathfrak{B} \in \mathbb{R}^{m_1 \times m_2 \times m_3}$ and  $\mathfrak{C} \in \mathbb{R}^{n_1 \times n_2 \times n_3}$, their Kronecker product is the tensor of $\mathbb{R}^{m_1 n_1 \times m_2 n_2 \times m_3 n_3}$ whose entries are
$$ [\mathfrak{B} \otimes \mathfrak{C}]_{k_1,k_2,k_3} = [\mathfrak{B}]_{i_1,i_2,i_3} [\mathfrak{C}]_{j_1,j_2,j_3} \quad \text{with} \quad k_d = j_d + (i_d-1)n_d \quad \text{for} \quad d=1,2,3. $$

In this paper we adopt the notation of \cite{Oseledets2009} and
say that a vector $\vect{v}\in\mathbb{R}^{n_1n_2n_3}$ is in Tucker format (or that it is a Tucker vector) when it is written as
\begin{equation} \label{eq:Tucker_vector} \vect{v} = \sum_{r_3=1}^{R_3^{\vect{v}}} \sum_{r_2=1}^{R_2^{\vect{v}}} \sum_{r_1=1}^{R_1^{\vect{v}}}  [\mathfrak{v}]_{r_1,r_2,r_3} v_{(3,r_3)} \otimes v_{(2,r_2)} \otimes v_{(1,r_1)}, \end{equation}
where $v_{(d,r_d)} \in \mathbb{R}^{n_d}$ for $d=1,2,3$ and $r_d = 1,\ldots,R_d^{\vect{v}}$, , while $\mathfrak{v} \in \mathbb{R}^{R_1^{\vect{v}} \times R_2^{\vect{v}} \times R_3^{\vect{v}}}$ is called the core tensor of $\vect{v}$. Here the triplet $\left( R_1^{\vect{v}},R_2^{\vect{v}},R_3^{\vect{v}}\right) $ is referred to as the multilinear rank of $\vect{v}$. Note that $\vect{v} \in \mathbb{R}^{n_1 n_2 n_3}$, and the memory required to store it as a Tucker vector is $R_1^{\vect{v}} R_2^{\vect{v}} R_3^{\vect{v}} + R_1^{\vect{v}} n_1 + R_2^{\vect{v}} n_2 + R_3^{\vect{v}} n_3$. This can be significantly lower than the standard $n_1 n_2 n_3$, provided that $R_d^{\vect{v}} \ll n_d$, for $d=1,2,3$.

Similarly, we say that a matrix $\matb{B}\in\mathbb{R}^{m_1 m_2 m_3\times n_1n_2n_3}$ is in Tucker format (or that it is a Tucker matrix) when it is written as
\begin{equation} \label{eq:Tucker_matrix} \matb{B} = \sum_{r_3=1}^{R_3^{\matb{B}}}  \sum_{r_2=1}^{R_2^{\matb{B}}} \sum_{r_1=1}^{R_1^{\matb{B}}}  [\mathfrak{B}]_{r_1,r_2,r_3} B_{(3,r_3)} \otimes B_{(2,r_2)} \otimes B_{(1,r_1)}, \end{equation}
where $B_{(d,r_d)} \in \mathbb{R}^{m_d \times n_d}$ for $d=1,2,3$, $r_d = 1,\ldots,R_d^{\matb{B}}$, and $\mathfrak{B} \in \mathbb{R}^{R_1^{\matb{B}} \times R_2^{\matb{B}} \times R_3^{\matb{B}}}$. 

The matrix-vector product between a Tucker matrix $\matb{B}$ as in \eqref{eq:Tucker_matrix} and a Tucker vector $\vect{v}$ as in \eqref{eq:Tucker_vector} can be efficiently computed since
\begin{align*} \matb{B} \vect{v} & = \sum_{r^{'}_1,r^{'}_2,r^{'}_3} \sum_{r^{''}_1,r^{''}_2,r^{''}_3} [\mathfrak{B}]_{r^{'}_1,r^{'}_2,r^{'}_3} [\mathfrak{v}]_{r^{''}_1,r^{''}_2,r^{''}_3} \left( B_{(3,r^{'}_3)} v_{(3,r^{''}_3)} \right)\otimes \left( B_{(2,r^{'}_2)} v_{(2,r^{''}_2)} \right) \otimes \left( B_{(1,r^{'}_1)} v_{(1,r^{''}_1)} \right) \\ & = \sum_{r_3=1}^{ R_3^{\matb{B}} R_3^{\vect{v}}} \sum_{r_2=1}^{ R_2^{\matb{B}} R_2^{\vect{v}}} \sum_{r_1=1}^{ R_1^{\matb{B}} R_1^{\vect{v}}}  [\mathfrak{B} \otimes \mathfrak{C}]_{r_1,r_2,r_3} \left( B_{(3,r^{'}_3)} v_{(3,r^{''}_3)} \right)\otimes \left( B_{(2,r^{'}_2)} v_{(2,r^{''}_2)} \right) \otimes \left( B_{(1,r^{'}_1)} v_{(1,r^{''}_1)} \right) \end{align*}
where, in the last line, $r^{'}_d$ and $r^{''}_d$ satisfy $r_d = r^{''}_d + \left(r^{'}_d - 1 \right) n_d $, for $d=1,2,3$.
Note that $\matb{B} \vect{v}$ is still a Tucker vector with multilinear rank $\left( R_1^{\matb{B}} R_1^{\vect{v}}, R_2^{\matb{B}} R_2^{\vect{v}}, R_3^{\matb{B}} R_3^{\vect{v}}\right) $.

Another operation where the Tucker structure can be exploited is the scalar product. Let $\vect{v}$ as in \eqref{eq:Tucker_vector} and let 
\begin{equation} \label{eq:Tucker_vector2} \vect{w} =  \sum_{r_3=1}^{R_3^{\vect{w}}}  \sum_{r_2=1}^{R_2^{\vect{w}}} \sum_{r_1=1}^{R_1^{\vect{w}}} [\mathfrak{w}]_{r_1,r_2,r_3} w_{(3,r_3)} \otimes w_{(2,r_2)} \otimes w_{(1,r_1)}, \end{equation}
where $w_{(d,r_d)} \in \mathbb{R}^{n_d}$ for $d=1,2,3$ and $r_d = 1,\ldots,R_d^{\vect{w}}$. Then
$$ \vect{v} \cdot \vect{w} = \sum_{r_3=1}^{R_3^{\vect{v}}}\sum_{r_2=1}^{R_2^{\vect{v}}} \sum_{r_1=1}^{R_1^{\vect{v}}}  \sum_{r'_3=1}^{R_3^{\vect{w}}}  \sum_{r'_2=1}^{R_2^{\vect{w}}}  \sum_{r'_1=1}^{R_1^{\vect{w}}}  [\mathfrak{v}]_{r_1,r_2,r_3}     [\mathfrak{w}]_{r^{'}_1,r^{'}_2,r^{'}_3} 
\prod_{d=1}^3 v_{(d,r_d)} \cdot w_{(d,r^{'}_d)} %= \sum_{r_1,r_2,r_3}[\mathfrak{v}]_{r_1,r_2,r_3} \left( \sum_{r^{'}_1,r^{'}_2,r^{'}_3}  [\mathfrak{w}]_{r^{'}_1,r^{'}_2,r^{'}_3} \prod_{d=1}^3 v_{(d,r_d)} \cdot w_{(d,r^{'}_d)} \right) . 
$$
The sum between two Tucker vectors is still a Tucker vector with multilinear rank equal to the sum of the multilinear ranks of the addends. Precisely, let $\vect{v}$ as in  \eqref{eq:Tucker_vector} and $\vect{w}$ as in \eqref{eq:Tucker_vector2}. Then
$$ \vect{z} = \vect{v} + \vect{w} = \sum_{r_3=1}^{R_3^{\vect{v}}+R_3^{\vect{w}}} \sum_{r_2=1}^{R_2^{\vect{v}}+R_2^{\vect{w}}} \sum_{r_1=1}^{R_1^{\vect{v}}+R_1^{\vect{w}}}  \mathfrak{z}_{r_1,r_2,r_3} z_{(3,r_3)} \otimes z_{(2,r_2)} \otimes z_{(1,r_1)}, $$
where
$$ z_{(d,r_d)} = \left\lbrace \begin{array}{ll} v_{(d,r_d)} & \text{for } r_d=1,\ldots,R_d^{\vect{v}} \\ w_{(d,r_d-R_d^{\vect{v}})} & \text{for } r_d=R_d^{\vect{v}} + 1,\ldots,R_d^{\vect{v}}+R_d^{\vect{w}}\end{array} \right. , \qquad d=1,2,3, $$
and where the core tensor $\mathfrak{z}$ is a block diagonal tensor defined by concatenating the core tensors of $\vect{v}$ and $\vect{w}$.
 
 \section{Low-rank multipatch method}
 \label{sec:lr_multipatch}

We introduce another subdivision of the computational domain, namely
let $\Theta^{(1)}, \ldots, \Theta^{(\mathcal{N}_{sub})}$
  be closed subsets of $  \Omega $  such that
\begin{equation} \label{eq:subdomains} \text{int} \left( \Omega \right) = \bigcup_{i = 1}^{\mathcal{N}_{sub}} \text{int} \left( \Theta^{(i)} \right). \end{equation}
where $ \text{int} \left( \cdot \right) $ denotes the interior of a set.
Here, differently from the patches defined in Section \ref{sec:iso_spaces}, the subdomains $\Theta^{(j)}$ are allowed to overlap, i.e. the intersection of two different subdomains can have a non-empty interior.

We denote with $V^{(i)}_{sub}$ the subspace of $V_h$ whose functions vanish outside $\Theta^{(i)}$, i.e.
$$ V^{(i)}_{sub} = \left\lbrace v_h \in V_h \middle| \; v_h = 0 \text{ on } \Omega \setminus \Theta^{(i)} \right\rbrace. $$
Note that \eqref{eq:subdomains} guarantees that 
$ V_h =  \bigcup_{i = 1}^{\mathcal{N}_{sub}} V^{(i)}_{sub} .$

As a basis for $V^{(i)}_{sub}$, we choose the set of basis functions of $V_h$ whose support is included in $\Theta^{(i)}$.
A crucial assumption is that each space $V^{(i)}_{sub}$ can be written as (the pushforward of) a tensor product spline space. { Precisely, we make the following assumption.
%This assumption must hold also for all the non-trivial intersection subspaces $V^{(i)}_{sub} \cap V^{(j)}_{sub}$.

\begin{ass}
 \label{ass:sub}
 We assume that
\begin{equation} \label{eq:Vjsub}
V_{sub}^{(i)} = \left\lbrace \widehat{v}_h \circ \mathcal{G}_{i}^{-1} \middle| \; \widehat{v}_h \in \widehat{\boldsymbol{\S}}_{sub,0}^{(i)} \right\rbrace ,
\end{equation}
where $\mathcal{G}_{i} :[0,1]^3 \longrightarrow \Theta^{(i)} $ is a non-singular and differentiable map, while 
\begin{equation} \label{eq:S_sub}\widehat{\boldsymbol{\S}}_{sub,0}^{(i)} = \widehat{\S}_{sub,1}^{(i)} \otimes \widehat{\S}_{sub,2}^{(i)} \otimes \widehat{\S}_{sub,3}^{(i)} \end{equation} 
is a tensor product spline space. 
\end{ass}}

Here $\widehat{\S}_{sub,d}^{(i)}$, $d=1,2,3$, represent univariate spline spaces, whose dimension is denoted with $n_{sub,d}^{(i)}$. 
Note that $\dim \left( V_{sub}^{(i)}\right) = \prod_{d=1}^{3} n_{sub,d}^{(i)} =: n_{sub}^{(i)}$.

In the setting described in Section \ref{sec:iso_spaces}, the subdomains $\Theta^{(i)}$ and the corresponding spaces $V_{sub}^{(i)}$ satisfying \eqref{eq:Vjsub} can be obtained by merging isogeometric patches
as well as the corresponding spaces. 
%(MENZIONO IL FATTO CHE CI SONO ASSUNZIONI SULLE BC?). 
We give details for this construction in Section \ref{sec:ptc_union}.

%The tensor-product assumption stated above, however, imposes some restriction on the boundary conditions that we can consider. Specifically, we exclude the case when a face with Dirichlet boundary conditions, belonging to a given patch, shares an edge with a face with Neumann boundary conditions belonging to another patch. We also exclude the case when two patches, whose intersections is one edge, have Neumann boundary conditions on the faces intersecting on that edge.

%OPPURE

%The tensor-product assumption stated above, however, does not hold for an arbitrary choice of the boundary conditions. Specifically, it does not hold when a face with Dirichlet boundary conditions, belonging to a given patch, shares an edge with a face with Neumann boundary conditions belonging to another patch. It also does not hold when two patches, whose intersections is one edge, have Neumann boundary conditions on the faces intersecting on that edge.

%These cases require a special treatment, which is not discussed here.

\subsection{The linear system}

We split the solution $\underline{u}_h \in \left[ V_h\right]^3$ of the
Galerkin problem \eqref{eq:galerkin} as
\begin{equation}
  \label{eq:solution_split}
  \underline{u}_h = \sum_{j}^{\mathcal{N}_{sub}} \underline{u}_h^{(j)}, \qquad \underline{u}_h^{(j)} \in \left[  V^{(j)}_{sub} \right]^3.
\end{equation}

The functions $
\underline{u}_h^{(1)}, \ldots, \underline{u}_h^{(\mathcal{N}_{sub})}$ represent the unknowns. Due to
the overlap between the subdomains, the subspaces $\left[
  V^{(1)}_{sub} \right]^3,\ldots,\left[ V^{(\mathcal{N}_{sub})}_{sub}
\right]^3$ are not necessarily in direct sum. As a result, 
$ \underline{u}_h^{(1)}, \ldots, \underline{u}_h^{(\mathcal{N}_{sub})}$
are not uniquely determined in general. Instead, they are selected by the Krylov
iterative solver, based on the initial iterate, without the need to
impose any average condition.

In order to set up the linear system, \B we consider the same splitting for the test functions, and on each space $\left[ V^{(j)}_{sub} \right]^3$ we choose as basis the set
$$ \left\lbrace \underline{e}_k B_i^{(j)} \; \middle| \; B_i^{(j)} \text{ basis function for } V^{(j)}_{sub}, \; k=1,2,3 \right\rbrace, $$
where $\underline{e}_k$ is the $k-$th vector of the canonical basis of
$\mathbb{R}^3$, $k=1,2,3$. Therefore, { 
  the solution of \eqref{eq:galerkin}  can be obtained
  from \eqref{eq:solution_split} and the linear system:}
%If we now consider a test function belonging to $V^{(i)}_{sub}$, $i=1,\ldots,N_{\Theta}$ the problem ?? can be written as
\begin{equation} \label{eq:lin_sys}
\matb{A} \vect{u} = \begin{bmatrix} \matb{A}^{(1,1)} & \ldots & \matb{A}^{(1,\mathcal{N}_{sub})} \\ \vdots & & \vdots \\ \matb{A}^{(\mathcal{N}_{sub},1)} & \ldots & \matb{A}^{(\mathcal{N}_{sub},\mathcal{N}_{sub})} \end{bmatrix} \begin{bmatrix} \vect{u}^{(1)} \\ \vdots \\ \vect{u}^{(\mathcal{N}_{sub})} \end{bmatrix} = \begin{bmatrix} \vect{f}^{(1)} \\ \vdots \\ \vect{f}^{(\mathcal{N}_{sub})} \end{bmatrix}
= \vect{f},
\end{equation}
where, for $i,j=1,\ldots,\mathcal{N}_{sub}$, $ \vect{u}^{(j)}$ is the coordinate vector of $\underline{u}^{(j)}_h$, while $\vect{f}^{(i)} \in \mathbb{R}^{3 n_{sub}^{(i)}}$ and $ \matb{A}^{(i,j)} \in \mathbb{R}^{3 n_{sub}^{(i)} \times 3 n_{sub}^{(j)}} $ are respectively the vector and matrix representations of the functional \eqref{eq:F} and of the bilinear form \eqref{eq:bil}, when considering $\left[ V^{(i)}_{sub} \right]^3 $ as test function space and $\left[ V^{(j)}_{sub} \right]^3$ as trial function space.
Note that the vectorial nature of the basis functions naturally induces a further block structure in $\matb{A}^{(i,j)}$ and $\vect{f}^{(i)}$:
$$ \matb{A}^{(i,j)} = \begin{bmatrix} \matb{A}^{(i,j,1,1)} & \matb{A}^{(i,j,1,2)} & \matb{A}^{(i,j,1,3)} \\ \matb{A}^{(i,j,2,1)} & \matb{A}^{(i,j,2,2)} & \matb{A}^{(i,j,2,3)} \\ \matb{A}^{(i,j,3,1)} & \matb{A}^{(i,j,3,2)} & \matb{A}^{(i,j,3,3)} \end{bmatrix}, \qquad \vect{f}^{(i)} = \begin{bmatrix} \vect{f}^{(i,1)} \\ \vect{f}^{(i,2)} \\ \vect{f}^{(i,3)} \end{bmatrix}, $$
with $\matb{A}^{(i,j,k,\ell)} \in \mathbb{R}^{n_{sub}^{(i)} \times n_{sub}^{(j)}} $ and $\vect{f}^{(i,k)} \in \mathbb{R}^{n_{sub}^{(i)}} $ for $k,\ell = 1,2,3$.

%Collecting these equations for all values of the index $i$, we obtain the linear system 
%$$ A u = b $$ 
%$$ A = \begin{bmatrix}  1 \end{bmatrix}, \qquad  $$
The linear system \eqref{eq:lin_sys} may be singular. Nevertheless, it
is solvable (as there exists a solution of the Galerkin problem) and
since $\matb{A}$ is symmetric and positive semidefinite one solution
could be found by a conjugate gradient
method, see  \cite{Kaasschieter1988}.% and in particular the work in
% preparation  \cite{BMS-in-prep}.

As a further step, in the spirit of \cite{Montardini2023}, we consider low-rank Tucker approximations of matrix and vector blocks, i.e.
\begin{equation} \label{eq:Ablock_Tucker} \matb{A}^{(i,j,k,\ell)} \approx \widetilde{\matb{A}}^{(i,j,k,\ell)} = \sum_{r_3=1}^{R_3^\matb{A}} \sum_{r_2=1}^{R_2^\matb{A}} \sum_{r_1=1}^{R_1^\matb{A}} \mathfrak{A}_{r_1,r_2,r_3} \widetilde{A}^{(i,j,k,\ell)}_{(3,r_3)} \otimes \widetilde{A}^{(i,j,k,\ell)}_{(2,r_2)} \otimes \widetilde{A}^{(i,j,k,\ell)}_{(1,r_1)}, \qquad %i,j = 1,\ldots,\mathcal{N}_{sub}, \; k,\ell = 1,2,3,  
\end{equation}
\begin{equation} \label{eq:fblock_Tucker} \vect{f}^{(i,k)} \approx \widetilde{\vect{f}}^{(i,k)} = \sum_{r_3=1}^{R_3^{\vect{f}}} \sum_{r_2=1}^{R_2^{\vect{f}}} \sum_{r_1=1}^{R_1^{\vect{f}}} \mathfrak{f}_{r_1,r_2,r_3} \widetilde{f}^{(i,k)}_{(3,r_3)} \otimes f^{(i,k)}_{(2,r_2)} \otimes f^{(i,k)}_{(1,r_1)}, 
%\qquad  i = 1,\ldots,\mathcal{N}_{sub}, \; k = 1,2,3,
\end{equation} 
for $i,j = 1,\ldots,\mathcal{N}_{sub}$, $k,\ell = 1,2,3$, where $R^{\matb{A}}_d = R^{\matb{A}}_d(i,j,k,\ell)$ and $R^{\vect{f}}_d = R^{\vect{f}}_d(i,k)$ for $d = 1,2,3$. In Section \ref{sec:ptc_union}
 we give details on how these approximations can be performed, when the subdomains are defined as unions of neighboring patches.
 We emphasize that, for the proposed strategy to be beneficial,
the multilinear ranks of the approximations, which depend on the geometry coefficients and on the forcing term, have to be reasonably small.

We collect all the newly defined matrix blocks $\widetilde{\matb{A}}^{(i,j,k,\ell)}$ and  vector blocks $\widetilde{\vect{f}}^{(i,k)}$ into a single matrix $\widetilde{\matb{A}}$ and a single vector $\widetilde{\vect{f}}$, respectively. Then we look for a vector $\widetilde{\vect{u}}$ with blocks in Tucker format that approximately solves the linear system
\begin{equation}
\label{eq:low_rank_sys}
 \widetilde{\matb{A}} \widetilde{\vect{u}} = \widetilde{\vect{f}}. 
\end{equation}
Following \cite{Montardini2023}, we tackle the above problem using a variant of the truncated preconditioned conjugate gradient (TPCG) method \cite{Kressner2011,Tobler2012}. This is similar to the standard preconditioned conjugate gradient, but after each step a truncation is applied to the vector blocks in order keep their ranks sufficiently small.

\subsection{The preconditioner}%\marginnote{citiamo anche Joaquin per lo scaling costante di (3.6)?} 

As an ``ideal'' preconditioner for $\widetilde{\matb{A}}$ we consider the block diagonal matrix
$$ \matb{P} = \begin{bmatrix} \matb{P}^{(1)} & 0 & \ldots & 0 \\ 0 & \matb{P}^{(2)} & \ldots & 0 \\ \vdots & \vdots & \ddots & \vdots \\ 0 & 0 & \ldots &  \matb{P}^{(\mathcal{N}_{sub})} \end{bmatrix}, $$
where for each $i=1,\ldots,\mathcal{N}_{sub}$, ${\matb{P}}^{(i)}$ is a block diagonal approximation of $\matb{A}^{(i,i)}$, i.e.
$$ \matb{P}^{(i)} = \begin{bmatrix} \matb{P}^{(i,1)} & 0 & 0 \\ 0 & \matb{P}^{(i,2)} & 0 \\ 0 & 0 & \matb{P}^{(i,3)}  \end{bmatrix}, $$
where $\matb{P}^{(i,k)}$ is an approximation of $\matb{A}^{(i,i,k,k)}$, $k=1,2,3$. More precisely, we choose 
\begin{equation} \label{eq:prec_block} \matb{P}^{(i,k)} = c^{(i,k)}_{1} M_3^{(i)} \otimes M_2^{(i)} \otimes K_1^{(i)} + c^{(i,k)}_{2} M_3^{(i)} \otimes K_2^{(i)} \otimes M_1^{(i)} + c^{(i,k)}_{3} K_3^{(i)} \otimes M_2^{(i)} \otimes M_1^{(i)}, \end{equation}
where, for $d=1,2,3,$ $K_{d}^{(i)}$ and $M_{d}^{(i)}$ are respectively
the stiffness and mass matrix over the univariate spline space
$\widehat{\S}_{sub,d}^{(i)}$. Differently from
\cite{Montardini2023}, here we introduce some constant coefficients
$c^{(i,k)}_{\ell} \in \mathbb{R}$ that aim to take into account the
geometry and the elasticity coefficients, whose choice is discussed
below. We emphasize that a similar approach was used in
\cite{Fuentes2023}.

Note that if $c^{(i,k)}_{\ell} = 1$ for every $\ell=1,2,3$, then $\matb{P}^{(i,k)}$ represents the discretization of the Laplace operator over the spline space $\widehat{\boldsymbol{\S}}_{sub,0}^{(i)}$.

%In practice, \marginnote{Io direi direttamente che il precondizione Ã¨
%un'approssimazione LR di quanto sopra, che non chiamiamo
%precondizione} $\left( \matb{P}^{(i,k)} \right)^{-1}$ is approximated
%with a low-rank Tucker matrix  following the steps described in
%\cite[Section 4.2]{Montardini2023}.

 We recall that, as discussed in \cite[Theorem 2]{BMS-arXiv}, the
use of a block diagonal preconditioner is equivalent to the use of additive overlapping
Schwarz preconditioners for the nonsingular linear system obtained by
constructing a global spline basis. Here, however, \B each diagonal block $\left( \matb{P}^{(i,k)} \right)^{-1}$ is further approximated with a low-rank Tucker matrix following the steps described in \cite[Section 4.2]{Montardini2023}.
We emphasize in particular that this strategy is not spoiled by the presence of the constant coefficients.
A key feature of this approach is that the action of $ \left( \widetilde{\matb{P}}^{(i,k)} \right)^{-1}$ on a vector can be computed using a variant of the fast diagonalization method \cite{Lynch1964,Sangalli2016} that
exploits the fast Fourier transform (FFT), yielding almost linear complexity. In conclusion, the preconditioner $\widetilde{\matb{P}}$ is the block diagonal matrix obtained by collecting all matrices $\widetilde{\matb{P}}^{(i,k)}$, for $i=1,\ldots,\mathcal{N}_{sub}$, $k=1,2,3$.

We now discuss the choice of the constant coefficients $c^{(i,k)}_{\ell}$ appearing in \eqref{eq:prec_block}. We fix $i \in \left\lbrace 1,\ldots,\mathcal{N}_{sub} \right\rbrace $ and $k \in \left\lbrace 1,2,3 \right\rbrace $. Let $v_h \in V_{sub}^{(i)}$ and let $\vect{v}$ denote its representing vector. It holds
$$ \vect{v}^T \matb{A}^{(i,i,k,k)} \vect{v} =\int_{\Theta^{(i)}}  2 \mu  \varepsilon(\underline{e}_{k} v_h) : \varepsilon(\underline{e}_{k} v_h) + \lambda  \left( \nabla \cdot \underline{e}_k v_h \right) \left(\nabla \cdot \underline{e}_{k} v_h \right) \d{\underline{x}} = \int_{[0,1]^3} \left( \nabla \widehat{v}_{sub}^{(i)} \right)^T Q^{(i,k)} \nabla \widehat{v}_{sub}^{(i)} \d{\underline{\xi}}, $$
where $\widehat{v}_{sub}^{(i)} = v_h \circ \mathcal{G}_i \in \widehat{\boldsymbol{\S}}^{(i)}_{sub,0}$, and
$$ Q^{(i,k)} = \vert \det \left( J_{\mathcal{G}_i} \right) \vert J_{\mathcal{G}_i}^{-1} \left[  \mu \left( I_3 + \underline{e}_{k} \underline{e}_k^T \right) + \lambda  
\underline{e}_{k} \underline{e}_{k}^T
\right] J_{\mathcal{G}_i}^{-T}.$$
Here $I_3$ denotes the $3 \times 3$ identity matrix. Moreover,
$$ \vect{v}^T \matb{P}^{(i,k)} \vect{v} = \int_{[0,1]^3} \left( \nabla \widehat{v}_{sub}^{(i)} \right)^T \begin{bmatrix} c^{(i,k)}_{1} & 0 & 0 \\ 0 & c^{(i,k)}_{2} & 0 \\ 0 & 0 & c^{(i,k)}_{3} \end{bmatrix} \nabla \widehat{v}_{sub}^{(i)} \d{\underline{\xi}}. $$
It is apparent that the $3 \times 3$ constant matrix appearing in the above formula should be chosen as the best constant diagonal approximation of $Q^{(i,k)}$. It is therefore reasonable to require that 
$$ c^{(i,k)}_{\ell} \approx Q^{(i,k)}_{\ell,\ell} (\xi_1,\xi_2,\xi_3), \qquad \ell=1,2,3, \quad \xi_1,\xi_2,\xi_3 \in [0,1]. $$
{ For $\ell = 1,2,3$, we  choose $c^{(i,k)}_{\ell}$ as the average of $Q^{(i,k)}_{\ell,\ell}$ evaluated on the breakpoints of the $l$-th   knot vector of the geometry that describes $\Theta^{(i)}$ and their midpoints. }

\subsection{The TPCG method with block-wise truncation}
\label{sec:tpcg}

As already mentioned, a crucial feature of the TPCG method is the truncation step. Indeed, the operations performed during the iterative process, namely the vector sums and more importantly the matrix-vector products, increase the multilinear ranks of the iterates. Therefore, it is of paramount importance to keep their rank reasonably small through truncation.
We emphasize that in our version of the TPCG method the involved vectors (and matrices) do not have a global Tucker structure. Instead, they can be subdivided into blocks with Tucker structure, each with its own multilinear rank. Therefore, all the operations that exploit the Tucker structure, including truncation, are performed to each matrix/vector block separately.

The simplest truncation operator considered here is relative truncation and it is denoted with $\trunc{{T}^{rel}}$. Given a Tucker vector $\vect{v}$ and a relative tolerance $\epsilon > 0$, $\vect{\widetilde{v}} = \trunc{{T}^{rel}}(\vect{v},\eta)$ is a Tucker vector with smaller or equal multilinear rank (in each direction) such that
\begin{equation} \label{eq:rel_trunc} \Vert \vect{\widetilde{v}} - \vect{v} \Vert_2 \leq \epsilon \Vert \vect{v} \Vert_2 .\end{equation}
We refer to \cite[Section 4.1.1]{Montardini2023} for more details. 

We also consider a dynamic truncation operator $\trunc{{T}^{dt}}$, which is used only for the approximate solution $\vect{u}_{k+1}$. Here a truncated approximation $\vect{\widetilde{u}}_{k+1}$ of $\vect{u}_{k+1}$ is initially computed using block-wise relative truncation with a prescribed initial tolerance $\epsilon$. Then the algorithm assesses how much the exact solution update $\Delta \vect{u}_k = \vect{u}_{k+1} - \vect{u}_{k}$ differs from the truncated one $\Delta \vect{\widetilde{u}}_k = \vect{\widetilde{u}}_{k+1} - \vect{u}_{k}$ by checking if the condition
$$ \left\vert \frac{\Delta \vect{\widetilde{u}}_k \cdot \Delta \vect{u}_k  }{\Vert \Delta \vect{u}_k  \Vert^2_2} - 1 \right\vert \leq \delta $$
is satisfied, where $\delta > 0$ is a prescribed tolerance. If the condition is satisfied, then $\vect{\widetilde{u}}_{k+1}$ is accepted as the next iterate. Otherwise, a new tolerance for the relative truncation is selected, equal to $\max \left\lbrace  \alpha \epsilon, \epsilon_{\min} \right\rbrace $, where $0 < \alpha < 1$ and $\epsilon_{\min} >0$ is a fixed minimum tolerance. The process is repeated until an acceptable truncated solution is found (or until the tolerance reaches $\epsilon_{\min}$). The final relative tolerance is returned and it is used as initial tolerance for $\trunc{{T}^{dt}}$ at the next TPCG iteration.  A starting relative tolerance $\epsilon_0$ is fixed at the beginning of the TPCG algorithm. We refer to  \cite[Section 4.1.2]{Montardini2023} for more details.

Except for $\vect{\widetilde{u}}_{k+1}$, all other vectors are compressed using block-wise relative truncation with tolerance 
$$ \eta_k = \beta \; tol \frac{ \Vert \vect{r}_0 \Vert_2}{\Vert \vect{r}_k \Vert_2},$$
where $\vect{r}_0$ and $\vect{r}_k$ are the initial and $k-$th residual vector, respectively, $0 < \beta < 1$ and \textit{tol} is the TPCG tolerance.
This choice of the tolerance promotes low rank, see also \cite{Simoncini2003,Palitta2021}.
Furthermore, we introduce intermediate truncation steps during the computation of matrix-vector products with $\matb{A}$, 
motivated by the significant rank increment yielded by this operation. 
%The tolerance $\gamma$ used for this intermediate truncation should be quite strict. Indeed, a loose tolerance here would lead the method to stagnation, as observed in
We remark, however, that intermediate truncation can lead to cancellation errors and stagnation, see e.g. the discussion in \cite[Section 3.6.3]{Tobler2012}. Therefore, a strict relative tolerance $\gamma$ should be chosen to safely perform this step.

%Furthermore, since the products with can signifi increase the rank of a vector due to the approximation of the , we introduced an intermediate truncation step during the. We remark that

The TPCG method is reported in Algorithm \ref{al:cg_tensor}, while Algorithm \ref{al:matvec} and Algorithm \ref{al:prec} perform respectively the matrix-vector products with the system matrix (and possibly the computation of the residual) and the application of the preconditioner. Note that, in the latter algorithms, each vector block is truncated immediately after being computed (for Algorithm \ref{al:matvec}, this is done in addition to the intermediate truncation steps discussed above).

%We assume that this approximation does not change the kernel of the matrix, i.e.
%$$ \ker (A) = \ker(\widetilde{A}) $$

\begin{rmk}{
If the magnitude of the solution varies  significantly in different
subdomains, it is a good idea to perform truncations on the vector
blocks such that the absolute error (instead of the relative error) is
equidistributed. This help in balancing the error among different subdomains. 
Therefore, given  a vector block $\vect{y}^{(i,k)} $, belonging to a
vector $\vect{y}$,  Its truncation $\widetilde{\vect{y}}^{(i,k)}$ could be chosen by imposing a condition of the form
$$\Vert \vect{\widetilde{y}}^{(i,k)} - \vect{y}^{(i,k)} \Vert_2 \leq \frac{\eta}{\mathcal{N}_{sub}} \Vert \vect{y} \Vert_2, $$
for a given $\eta > 0$ independent of $i \in \left\lbrace 1, \ldots, \mathcal{N}_{sub} \right\rbrace $, and $k \in \left\lbrace 1,2,3 \right\rbrace $. Note that in the right-hand side of the above inequality we consider the norm of $\vect{y}$, rather than that of its block $\vect{y}^{(i,k)}$. }

In the present paper we do not use this strategy since for all problems considered in Section \ref{sec:numerics}, the solution is not significantly larger on a portion of the domain with respect to the rest. 
\end{rmk}

\begin{algorithm}
\caption{TPCG}\label{al:cg_tensor}
\hspace*{\algorithmicindent} \textbf{Input}: { System matrix   $\matb{\widetilde{A}}$  and block diagonal preconditioner $\widetilde{\matb{P}}$ in block-wise Tucker format, right-hand side $\vect{\widetilde{f}}$ and initial guess $\vect{u}_0$ in block-wise Tucker format, TPCG tolerance $tol>0$, parameter $\beta$ for the relative truncation, intermediate relative truncation tolerance $\gamma$, parameters for the dynamic truncation: starting relative tolerance $\epsilon_0$, reducing factor $\alpha$, minimum tolerance $\epsilon_{\min}$, threshold $\delta$.} \\
 \hspace*{\algorithmicindent} \textbf{Output}: Low-rank solution $\widetilde{\vect{u}}$  of $\matb{\widetilde{A}} \widetilde{\vect{u}} =\vect{\widetilde{f}}.$
 \begin{algorithmic}[1]
 \State $\eta_0 = \beta \; tol$;
 \State $\vect{r}_0= -{\tt Matvec} ( \matb{\widetilde{A}}, \vect{u}_0, \vect{\widetilde{f}}, \eta_0 ) $
 %\trunc{{T}^{rel}}(\vect{\widetilde{f}}-\matb{\widetilde{A}}\vect{x}_0,\eta_0)$;
  \State $\vect{z}_0=  {\tt Prec} ( \matb{\widetilde{P}}, \vect{r}_0,\eta_0 )$;
  \State $\vect{p}_0=\vect{z}_0$;
  %\State $\vect{q}_0=\trunc{{T}^{rel}}(\matb{\widetilde{A}}\vect{p}_0,\eta_0)$;
 % \State ${\xi}_0=\vect{p}_0\cdot\vect{q}_0$;
  \State $k=0$
 \While{$\|\vect{r}_{k}\|_2>tol$}{\\
         \qquad  $\vect{q}_{k}= {\tt Matvec}(\matb{\widetilde{A}}, \vect{p}_{k},\vect{0},\gamma,\eta_{k})$; \\
         \qquad ${\xi}_{k}=\vect{p}_{k}\cdot \vect{q}_{k}$\\
         \qquad  $\omega_k = \frac{\vect{r}_k \cdot \vect{p}_k }{{\xi}_k}$;\\ 
         \qquad  $[\vect{u}_{k+1},\epsilon_{k+1}]={    \trunc{{T}^{dt}}}(\vect{u}_k, \vect{u}_k+\omega_k\vect{p}_k,\epsilon_k, \alpha, \epsilon_{\min},\delta)$; \\  
       %  \qquad  $\vect{r}_{k+1}={  \trunc{{T}^{rel}}}(\vect{\widetilde{f}} -\matb{\widetilde{A}}\vect{x}_{k+1},\eta_{k})$ ; \\
       \qquad $ \vect{r}_{k+1} =  -{\tt Matvec} (\matb{\widetilde{A}},\vect{u}_{k+1}, \vect{\widetilde{f}},\gamma,\eta_k)$ \\
         \qquad $\eta_{k+1}=\beta \; tol\frac{\|\vect{r}_0\|_2}{\|\vect{r}_{k+1}\|_2}$;\\
         \qquad  $\vect{z}_{k+1} = {\tt Prec} (\vect{r}_{k+1},\eta_{k+1})$; \\
         \qquad  $\beta_k=-\frac{\vect{z}_{k+1}\cdot\vect{q}_k} {\vect{ \xi}_{k}}$;\\
        \qquad    $\vect{p}_{k+1}=  \trunc{{T}^{rel}}(\vect{z_{k+1}}+\beta_k\vect{p}_k, \eta_{k+1})$;  \\
        \qquad   $k=k+1$;
  }\EndWhile
\State $\widetilde{\vect{u}}=\vect{u}_{k}.$
\end{algorithmic}
\end{algorithm}

\begin{algorithm}
\caption{{\tt Matvec}}\label{al:matvec}
  \hspace*{\algorithmicindent} \textbf{Input}: Matrix $\matb{\widetilde{A}}$ in block-wise Tucker format, vectors $\vect{x}$ and $\vect{b}$ in block-wise Tucker format, intermediate and final relative truncation tolerances $\gamma$ and $\eta$  \\
 \hspace*{\algorithmicindent} \textbf{Output}: vector $ \widetilde{\vect{y}}$, block-wise truncation of $\vect{y} = \matb{\widetilde{A}} \vect{x}- \vect{b}$
 \begin{algorithmic}[1]
 \For{ $i = 1,\ldots,\mathcal{N}_{sub}$ }{
        \qquad \For{ $k = 1,2,3$ }{\\
\qquad \quad $\widetilde{\vect{y}}^{(i,k)} = - \vect{b}^{(i,k)} $ \For{ $j = 1, \ldots, \mathcal{N}_{sub}$}{ \For{ $\ell=1,2,3$ }{ \\ \qquad \qquad \qquad $\widetilde{\vect{y}}^{(i,k)} = \trunc{{T}^{rel}} \left( \widetilde{\vect{y}}^{(i,k)} + \matb{\widetilde{A}}^{(i,j,k,\ell)} \vect{x}^{(j,\ell)}, \gamma \right) $} \EndFor } \EndFor
\\ \qquad \quad $\widetilde{\vect{y}}^{(i,k)} = \trunc{{T}^{rel}} \left( \widetilde{\vect{y}}^{(i,k)}, \eta \right)$} \EndFor } \EndFor
\end{algorithmic}
\end{algorithm}

\begin{algorithm}
\caption{{\tt Prec}}\label{al:prec}
  \hspace*{\algorithmicindent} \textbf{Input}: Block diagonal preconditioner $\matb{\widetilde{P}}$ with blocks in Tucker format, vector $\vect{r}$ in block-wise Tucker format, relative truncation tolerance $\eta$ \\
 \hspace*{\algorithmicindent} \textbf{Output}: vector $ \widetilde{\vect{z}}$, block-wise truncation of $\vect{z} = \matb{\widetilde{P}}^{-1} \vect{r}$
 \begin{algorithmic}[1]
 \For{ $i = 1,\ldots,\mathcal{N}_{sub}$ }{
        \For{ $k = 1,2,3$ }{\\
\qquad \quad $\widetilde{\vect{z}}^{(i,k)} = \trunc{{T}^{rel}} \left( \left(\matb{\widetilde{P}}^{(i,k)}\right)^{-1} \vect{r}^{(i,k)} , \eta \right) $
}\EndFor  }\EndFor
\end{algorithmic}
\end{algorithm}

\subsection{Subdomains as unions of patches}
\label{sec:ptc_union}

%(METTO PARTE O TUTTA QUESTA SEZIONE IN UN'APPENDICE?)

%In the setting described in Section \ref{sec:iso_spaces}, we can choose the subdomains as unions of pairs of patches sharing a face. More precisely, for each face shared by two patches, we introduce a subdomain defined as the union of those patches. Note that in this case $\mathcal{N}_{sub}$ is equal to the number of faces belonging to the interface $\Gamma$.

In the setting described in Section \ref{sec:iso_spaces}, we can choose the subdomains as unions of patches.
More precisely, the subdomains can be defined as follows. For every corner that is shared by 8 patches, we introduce a subdomain defined as the union of those 8 patches. Then, for every edge that it is shared by 4 patches and that is not contained in the interior of the previous subdomains, we introduce a new subdomain defined as the union of those 4 patches. Finally, for every face that is shared by 2 patches and that is not contained in the interior of the previous subdomains, we introduce a new subdomain defined as the union of those 2 patches. Note that this procedure guarantees that \eqref{eq:subdomains} is satisfied.

In the reminder of this section we
give details on the function subspaces and on the construction of the Galerkin matrices for the case 
of subdomains defined as the union of pairs of patches. The treatment of other subdomains represents a straightforward extension of this case. Indeed, a 4-patches subdomain can be treated as the union of a pair of 2-patches subdomains. Similarly, an 8-patches subdomain can be treated as the union of a pair of 4-patches subdomains.

%In the reminder of this section we discuss condition and of the Galerkin matrices for subdomains that are union of pairs of patches sharing a face. Other subdomains is a straightforward extension of this case. Indeed, a 4-patches subdomain can be constructed  

We fix $i\in \left\lbrace 1,\ldots,\mathcal{N}_{sub} \right\rbrace $ and let $i_1,i_2 \in \left\lbrace 1,\ldots,\mathcal{N}_{ptc}\right\rbrace $ such that $\Theta^{(i)} = \Omega^{(i_1)} \cup \Omega^{(i_2)}$.
%A basis for $V_{sub}^{(j)}$ is given by the set of basis functions of $V_h$ whose support is included in $\Omega^{(j_1)} \cup \Omega^{(j_2)}$.
%
%We now introduce an assumption on the boundary conditions. Precisely 
We assume that if a face of $\Omega^{(i_1)}$ shares an edge with a face of $\Omega^{(i_2)}$, then either both faces belong to $\Gamma \cup \partial\Omega_D$, or they both belong to $\partial\Omega_N $. This assumption guarantees that the space obtained by ``merging'' $V_{ptc}^{(i_1)} $ and $V_{ptc}^{(i_2)}$ can be written as (the pushforward of) a tensor product spline space. Below we give the details. 
%(MAGARI FARE UN REMARK SUL CASO GENERALE?)

%We consider, for the sake of the explanation, the case when the patches $\Omega^{(j_1)}$ and $\Omega^{(j_2)}$ are attached along the third direction.
%$$ \left\lbrace \mathcal{F}^{(j_1)}(\xi_1,\xi_2,1) \vert (\xi_1,\xi_2) \in [0,1]^2 \left\rbrace = \left\lbrace \mathcal{F}^{(j_2)}(\xi_1,\xi_2,0) \vert (\xi_1,\xi_2) \in [0,1]^2 \left\rbrace $$

%\begin{rmk}
%In the case of arbitrary boundary conditions, merging two local spline spaces does not necessarily results in a space which 

%In this case, one can define $V_{sub}^{(i)}$ by substituting any Neumann boundary conditions with homogenous Dirichlet.
%\end{rmk}

It is not restrictive to assume that
$\Omega^{(i_1)}$ and $\Omega^{(i_2)}$ are attached along the third parametric direction, and in particular that
$$ \mathcal{F}_{(i_1)}(\xi_1,\xi_2,1) = \mathcal{F}_{(i_2)}(\xi_1,\xi_2,0), \qquad (\xi_1,\xi_2) \in [0,1]^2. $$
Because of the patch conformity assumption, the parametric spline spaces $\widehat{\boldsymbol{\S}}_{ptc}^{(i_1)}$ and $\widehat{\boldsymbol{\S}}_{ptc}^{(i_2)}$ must have the same knot vectors in the first 2 directions, i.e. $\Xi_1^{(i_1)} = \Xi_1^{(i_2)} =: \Xi_1$ and $\Xi_2^{(i_1)} = \Xi_2^{(i_2)} =: \Xi_2$. Moreover, we consider the knot vectors in the third directions
$$ \Xi_3^{(i_1)} = \left\lbrace \underbrace{0,\ldots,0}_{p+1}, \xi^{(i_1)}_{p+2},\ldots,\xi^{(i_1)}_{m_3^{(i_1)}},\underbrace{1,\ldots,1}_{p+1} \right\rbrace, \qquad \Xi_3^{(i_2)} = \left\lbrace \underbrace{0,\ldots,0}_{p+1}, \xi^{(i_2)}_{p+2},\ldots,\xi^{(i_2)}_{m_3^{(i_2)}},\underbrace{1,\ldots,1}_{p+1} \right\rbrace.  $$

We now define the ``merged'' spline space $\widehat{\boldsymbol{\S}}_{sub}^{(i)}$, generated by the knot vectors $\Xi_1$, $\Xi_2$ and
$$ \Xi_3 = \left\lbrace \underbrace{0,\ldots,0}_{p+1}, \frac{1}{2}\xi^{(i_1)}_{p+2},\ldots,\frac{1}{2} \xi^{(i_1)}_{m_3^{(i_1)}}, \underbrace{\frac{1}{2},\ldots,\frac{1}{2}}_p, \frac{1}{2}\xi^{(i_2)}_{p+2}+ \frac{1}{2},\ldots,\frac{1}{2} \xi^{(i_2)}_{m_3^{(i_2)}}+ \frac{1}{2}, \underbrace{1,\ldots,1}_{p+1} \right\rbrace.$$
We also consider the geometry mapping $\mathcal{G}_{i}: [0,1]^3 \longrightarrow \Theta^{(i)} $
\begin{equation} \label{eq:mapG} \mathcal{G}_{i}\left( \xi_1,\xi_2,\xi_3 \right)  = \left\lbrace \begin{array}{ll} \mathcal{F}_{i_1}\left( \xi_1,\xi_2,2 \xi_3 \right)  & \text{if } 0 \leq \xi_3 \leq \frac{1}{2}, \\ \mathcal{F}_{i_2}\left( \xi_1,\xi_2,2\xi_3-1 \right) & \text{if } \frac{1}{2} < \xi_3 \leq 1. \end{array} \right.  \end{equation}
We introduce the subspace $\widehat{\boldsymbol{\S}}_{sub,0}^{(j)} \subseteq \widehat{\boldsymbol{\S}}_{sub}^{(j)} $ generated by the basis functions of $\widehat{\boldsymbol{\S}}_{sub}^{(j)}$ whose image through $\mathcal{G}_{j}$ vanishes on { $(\Gamma\setminus (\partial\Omega^{(i_1)}\cap \partial\Omega^{(i_2)})) \cup \partial \Omega_D$}. We emphasize that, under the aforementioned assumptions on the boundary conditions, $\widehat{\boldsymbol{\S}}_{sub,0}^{(i)}$ is a tensor product space as in \eqref{eq:S_sub}. %In particular, we can indicate the basis functions of this space using a multi-index $i = (i_1,i_2,i_3)$ as follows:
It holds
$$V_{sub}^{(i)} = \left\lbrace \widehat{v}_h \circ \mathcal{G}_{i}^{-1} \middle| \; \widehat{v}_h \in \widehat{\boldsymbol{\S}}_{sub,0}^{(i)} \right\rbrace. $$

%\subsection{Low-rank approximation of matrix and vector blocks}

\begin{rmk}\label{rmk:boundary_cond}
In the case of arbitrary boundary conditions, merging two local spline spaces does not necessarily results in a space which is the pushforward of a tensor product spline space {\R1 (see some examples for the L-shaped domain in Figure \ref{fig:not_allowed}).} 
  In this case, one can define $V_{sub}^{(i)}$ by substituting any Neumann boundary conditions with homogeneous Dirichlet. Then an additional subdomain should be introduced containing the Neumann degrees of freedom.  
  
  \begin{center}
  \begin{figure}
  \includegraphics[scale=0.27]{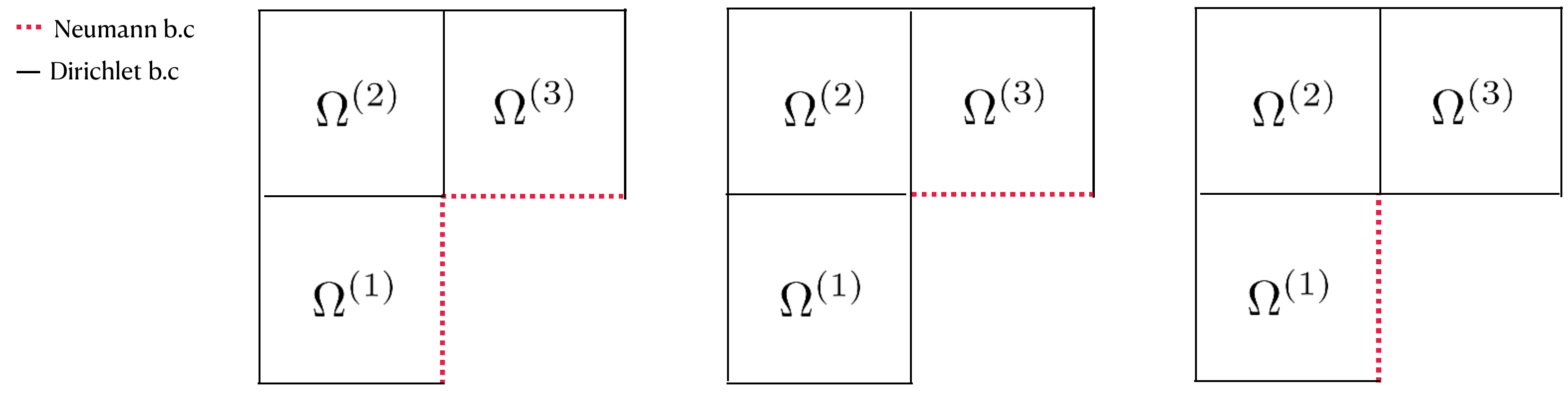}
\caption{Examples of boundary conditions that are not allowed for the L-shaped domain. A remedy is proposed in Remark \ref{rmk:boundary_cond}}\label{fig:not_allowed}
\end{figure}
  \end{center}
  
\end{rmk}

We now discuss in detail how the approximations  \eqref{eq:Ablock_Tucker} and \eqref{eq:fblock_Tucker} can be performed using low-rank techniques. We first consider the case of a matrix block $\matb{A}^{(i,j)}$ for $i \neq j$, $i,j \in \left\lbrace 1,\ldots,\mathcal{N}_{sub}\right\rbrace $. Of course if $\Theta^{(i)} \cap \Theta^{(j)} $ has an empty interior, then $\matb{A}^{(i,j)}$ is a null matrix. On the other hand, if $\Theta^{(i)} \cap \Theta^{(j)} $ has a non-empty interior, then there is an index $m \in \left\lbrace 1, \ldots, \mathcal{N}_{ptc} \right\rbrace $ such that $\Theta^{(i)} \cap \Theta^{(j)} = \Omega^{(m)}$. Let $v_h \in V^{(i)}_{sub}$ and $w_h \in V^{(j)}_{sub}$, and let $\vect{v}$ and $\vect{w}$ be their vector representations. 
Then for $k,\ell = 1,2,3$, it holds
\begin{align}  \begin{split} \label{eq:Ablock}
\vect{v}^T \matb{A}^{(i,j,k,\ell)} \vect{w} & =\int_{\Omega^{(m)}}  2 \mu \;  \varepsilon(\underline{e}_{k} v_h) : \varepsilon(\underline{e}_{\ell} w_h) + \lambda  \left( \nabla \cdot \underline{e}_k v_h \right) \left(\nabla \cdot \underline{e}_{\ell} w_h \right) \d{\underline{x}} = \\ & = \int_{[0,1]^3} \left( \nabla \widehat{v}_{ptc}^{(m)} \right)^T C^{(m,k,\ell)} \nabla \widehat{w}_{ptc}^{(m)} \d{\underline{\xi}}, \end{split} \end{align}
where $\widehat{v}_{ptc}^{(m)} = v_h  \circ \mathcal{F}_m, \widehat{w}_{ptc}^{(m)} = w_h \circ \mathcal{F}_m$, and
%\in \widehat{\boldsymbol{\S}}_{ptc,0}^{(m)} $
\begin{equation} \label{eq:Cmkl}
C^{(m,k,\ell)} = \vert \det \left( J_{\mathcal{F}_m} \right) \vert J_{\mathcal{F}_m}^{-1} \left[  \mu \left( \delta_{k,\ell} I_3 + \underline{e}_{\ell} \underline{e}_k^T \right) + \lambda  
\underline{e}_{k} \underline{e}_{\ell}^T
\right] J_{\mathcal{F}_m}^{-T}.\end{equation}
Here $\delta_{k,\ell}$ is the Kronecker delta. Note that $\widehat{v}_{ptc}^{(m)}, \widehat{w}_{ptc}^{(m)} \in \widehat{\boldsymbol{\S}}_{ptc,0}^{(m)}$. Note that 
if $\text{supp}\left( v_h \right) \cap \text{supp}\left( w_h \right) \cap \Omega^{(m)}$ has an empty interior, at least one between $\widehat{v}_{ptc}^{(m)}$ and $\widehat{w}_{ptc}^{(m)}$ is the null function, and therefore $\vect{v}^T \matb{A}^{(i,j,k,\ell)} \vect{w} = 0$.

We consider a low-rank approximation of the entries of $C^{(m,k,\ell)}$, i.e. 
\begin{equation} \label{eq:low_rank_C} C^{(m,k,\ell)}_{\alpha,\beta} \left( \xi_1,\xi_2,\xi_3 \right) \approx \widetilde{C}^{(m,k,\ell)}_{\alpha,\beta} \left( \xi_1,\xi_2,\xi_3 \right) = \sum_{r_1,r_2,r_3} \mathfrak{c}_{r_1,r_2,r_3}^{(m,k,\ell,\alpha,\beta)}   c^{(m,k,\ell,\alpha,\beta)}_{(1,r_1)} \left( \xi_1 \right)   c^{(m,k,\ell,\alpha,\beta)}_{(2,r_2)} \left( \xi_2 \right)    c^{(m,k,\ell,\alpha,\beta)}_{(3,r_3)} \left( \xi_3 \right),
\end{equation}
%\begin{equation} \label{eq:low_rank_C} C^{(m,k,\ell)}_{\alpha,\beta} \left( \xi_1,\xi_2,\xi_3 \right) \approx \sum_{r_1,r_2,r_3} c^{(m,k,\ell,\alpha,\beta)}_{r_1,r_2,r_3} \left( \xi_1,\xi_2,\xi_3 \right) , \qquad \alpha,\beta = 1,2,3 \end{equation}
%where
%$$ c^{(m,k,\ell,\alpha,\beta)}_{r_1,r_2,r_3} \left( \xi_1,\xi_2,\xi_3 \right) = \mathfrak{c}_{r_1,r_2,r_3}^{(m,k,\ell,\alpha,\beta)} \cdot c^{(m,k,\ell,\alpha,\beta)}_{1,r_1} \left( \xi_1 \right) \cdot  c^{(m,k,\ell,\alpha,\beta)}_{2,r_2} \left( \xi_2 \right) \cdot  c^{(m,k,\ell,\alpha,\beta)}_{3,r_3} \left( \xi_3 \right) $$
%Plugging \eqref{eq:low_rank_C} into \eqref{eq:Ablock}, we obtain
for $\alpha,\beta = 1,2,3,$ and we define $\widetilde{\matb{A}}^{(i,j,k,\ell)}$ by replacing $C^{(m,k,\ell)}$ with $\widetilde{C}^{(m,k,\ell)}$ into \eqref{eq:Ablock}, i.e.
%
%$$ \vect{v}^T \widetilde{\matb{A}}^{(i,j,k,\ell)} \vect{w} = \sum_{r_1,r_2,r_3} \sum_{\alpha,\beta=1}^3 \mathfrak{c}_{r_1,r_2,r_3}^{(m,k,\ell,\alpha,\beta)} \int_{[0,1]^3} c^{(m,k,\ell,\alpha,\beta)}_{(1,r_1)} \; c^{(m,k,\ell,\alpha,\beta)}_{(2,r_2)} \;  c^{(m,k,\ell,\alpha,\beta)}_{(3,r_3)}  \; \partial_{\alpha} \widehat{v}_{ptc}^{(m)} \; \partial_{\beta} \widehat{w}_{ptc}^{(m)} \; d \underline{\xi}. $$
\begin{equation} \label{eq:Atilde_block1} \vect{v}^T \widetilde{\matb{A}}^{(i,j,k,\ell)} \vect{w} = \int_{[0,1]^3} \left( \nabla \widehat{v}_{ptc}^{(m)} \right)^T \widetilde{C}^{(m,k,\ell)} \; \nabla \widehat{w}_{ptc}^{(m)} \; \d{\underline{\xi}}.  \end{equation}

Thanks to \eqref{eq:low_rank_C} and to the tensor structure of the spline space $\widehat{\boldsymbol{\S}}_{ptc,0}^{(m)}$, we have
\begin{equation} \label{eq:low_rank_Ablock}  \widetilde{\matb{A}}^{(i,j,k,\ell)} = \sum_{r_1,r_2,r_3} \sum_{\alpha,\beta=1}^3 \mathfrak{c}_{r_1,r_2,r_3}^{(m,k,\ell,\alpha,\beta)} \widetilde{A}^{(i,j,k,\ell,\alpha,\beta)}_{(3,r_3)} \otimes  \widetilde{A}^{(i,j,k,\ell,\alpha,\beta)}_{(2,r_2)} \otimes  \widetilde{A}^{(i,j,k,\ell,\alpha,\beta)}_{(1,r_1)}, \end{equation}
where each Kronecker factor $\widetilde{A}^{(i,j,k,\ell,\alpha,\beta)}_{(d,r_d)}$ is an $n^{(i)}_{sub,d} \times n^{(j)}_{sub,d}$ matrix with null entries outside of a square $n_{ptc,d}^{(m)} \times n_{ptc,d}^{(m)}$ block $\widetilde{A}^{(m,k,\ell,\alpha,\beta)}_{(d,r_d)}$  whose entries are
$$ \left( \widetilde{A}^{(m,k,\ell,\alpha,\beta)}_{(d,r_d)}\right)_{s,t}  = \int_0^1 c^{(m,k,\ell,\alpha,\beta)}_{(d,r_d)} (\xi_d) \left[   \widehat{b}_{d,s} ^{(m)} \left( \xi_d \right)\right]^{(d,\alpha)}  \left[  \widehat{b}_{d,t}^{(m)}\left( \xi_d \right)\right]^{(d,\beta)}   d \xi_d, \qquad s,t = 1,\ldots,n_{ptc,d}^{(m)}, $$
with
$$ \left[   f\left(  \xi \right)  \right]^{(d,\gamma)}  = \left\{ \begin{array}{cl}
\displaystyle f'\left( \xi \right)  & \text{if } d = \gamma   \\
\displaystyle f\left( \xi \right) & \text{otherwise} \end{array}
\right., \qquad d,\gamma=1,2,3,$$
Note that \eqref{eq:low_rank_Ablock} can be easily written in the Tucker format \eqref{eq:Ablock_Tucker}.
%We also observe, as it will be useful in the following

%A viable approach to compute the low-rank approximation \eqref{eq:low_rank_C} is to interpolate the left-hand side function 

There is of course more than one approach to compute the low-rank approximation \eqref{eq:low_rank_C}. A suitable possibility, proposed in \cite{Mantzaflaris2017}, is to interpolate the left-hand side function using a coarse spline space and then performing a High Order Singular Value Decomposition (HOSVD) \cite{De2000}. In this work, however, we follow \cite{Montardini2023} and employ the strategy described in \cite{Driscoll2014,Dolgov2021}, which is based on interpolation with Chebyshev polynomials. 
Note that approximation \eqref{eq:low_rank_C} has to be computed only for $\alpha \geq \beta$ and $k \geq \ell$, because of the symmetries in the coefficients.

We now discuss the low-rank approximation of a given vector block $\vect{f}^{(i,k)}$, where $i \in \left\lbrace 1,\ldots,\mathcal{N}_{ptc} \right\rbrace $, and $k \in \left\lbrace 1,2,3 \right\rbrace $. 
%As before, we denote with $i_1,i_2$ the indices such that $\Theta^{(i)} = \Omega^{(i_1)} \cup \Omega^{(i_2)} $. 
Let $v_h \in V^{(i)}_{sub}$ and let $\vect{v}$ be its vector representation. It holds
\begin{equation} \label{eq:f_block} \vect{v}^T \vect{f}^{(i,k)}  = \int_{\Theta^{(i)}} f_k (\underline{x}) \; v_h (\underline{x}) \; d \underline{x} = \int_{[0,1]^3} \widehat{f}_{k}^{(i)} (\underline{\xi}) \; \widehat{v}_{sub}^{(i)} (\underline{\xi}) \; d \underline{\xi}, \end{equation}
%\begin{align} \label{eq:fblock} \begin{split} v^T \vect{f}^{(i,k)}  = \int_{\Theta^{(i)}} f_k (\underline{x}) \; v_h (\underline{x}) \; d \underline{x} & =  \int_{\Omega^{(i_1)}} f_k (\underline{x}) \; v_h (\underline{x}) \; d \underline{x} + \int_{\Omega^{(i_2)}} f_k (\underline{x}) \; v_h (\underline{x}) \; d \underline{x} = \\  & = \int_{[0,1]^3} \omega^{(i_1)}(\underline{\xi}) \widehat{f}_{k}^{(i_1)} (\underline{\xi}) \; \widehat{v}_{h}^{(i_1)} (\underline{\xi}) \; d \underline{\xi} + \int_{[0,1]^3} \omega^{(i_2)}(\underline{\xi}) \widehat{f}_{k}^{(i_2)} (\underline{\xi}) \; \widehat{v}_{h}^{(i_2)} (\underline{\xi}) \; d \underline{\xi} \end{split} \end{align}
where $\widehat{f}_{k}^{(i)} = \det \left( J_{\mathcal{G}_i}\right) \left( f_k \circ \mathcal{G}_{i} \right)$ and $ \widehat{v}_{sub}^{(i)} = v_h \circ \mathcal{G}_{i} \in \widehat{\boldsymbol{\S}}_{sub,0}^{(i)} $.
%where $\widehat{f}_{k}^{(m)} = \det \left( J_{\mathcal{F}_m}\right) \left( f_k \circ \mathcal{F}_{m} \right)$ and $ \widehat{v}_{ptc}^{(m)} = v_h \circ \mathcal{F}_{m} \in \widehat{\boldsymbol{\S}}_{ptc,0}^{(m)} $, for $m=i_1,i_2$.
We consider a low-rank approximation
\begin{equation} \label{eq:low_rank_f}  \widehat{f}_{k}^{(i)} \left( \xi_1,\xi_2,\xi_3 \right) \approx \widetilde{f}_{k}^{(i)}\left( \xi_1,\xi_2,\xi_3 \right)    =  \sum_{r_1,r_2,r_3}  \mathfrak{f}^{(i,k)}_{r_1,r_2,r_3}   f^{(i,k)}_{(1,r_1)} \left( \xi_1 \right)   f^{(i,k)}_{(2,r_2)} \left( \xi_2 \right)   f^{(i,k)}_{(3,r_3)} \left( \xi_3 \right),
 \end{equation}
%where $\omega^{(I)} = \det \left( J_{\mathcal{F}_{I}}\right) $, $\widehat{f}_{k}^{(I)} = f_k \circ \mathcal{F}_{I}^{-1}$ and $ \widehat{v}_{h}^{(I)} = v_h \circ \mathcal{F}_{I}^{-1} $ for $I = i_1,i_2$. We consider a low-rank approximation
%$$ \omega^{(I)} \left( \xi_1,\xi_2,\xi_3 \right) \widehat{f}_k^{(I)} \left( \xi_1,\xi_2,\xi_3 \right) \approx \sum_{r_1,r_2,r_3} \mathfrak{c}^{(I,k)}_{r_1,r_2,r_3}  \cdot f^{(I,k)}_{(1,r_1)} \left( \xi_1 \right) \cdot f^{(I,k)}_{(2,r_2)} \left( \xi_2 \right) \cdot f^{(I,k)}_{(3,r_3)} \left( \xi_3 \right), \qquad I = i_1,i_2 $$ 
 
We define $\widetilde{\vect{f}}^{(i,k)}$ by replacing $\widehat{f}_{k}^{(i)}$ with $\widetilde{f}_{k}^{(i)}$ into \eqref{eq:f_block}, i.e.
$$ \vect{v}^T \widetilde{\vect{f}}^{(i,k)} = \int_{[0,1]^3} \widetilde{f}_{k}^{(i)} \; \widehat{v}_{sub}^{(i)}(\underline{\xi}) \; \d \underline{\xi} = \sum_{r_1,r_2,r_3} \mathfrak{f}^{(i,k)}_{r_1,r_2,r_3} \int_{[0,1]^3} f^{(i,k)}_{(1,r_1)} \left( \xi_1 \right) \; f^{(i,k)}_{(2,r_2)} \left( \xi_2 \right) \; f^{(i,k)}_{(3,r_3)} \left( \xi_3 \right) \; \widehat{v}_{sub}^{(i)}(\underline{\xi}) \; \d \underline{\xi}.$$

Therefore, thanks to the tensor structure of the spline space $\widehat{\boldsymbol{\S}}_{sub,0}^{(i)}$ the vector block $\vect{f}^{(i,k)}$ can be approximated as in \eqref{eq:fblock_Tucker}. 

%$$ \widetilde{f}^{(i,k)}_{(d,r_d)} = \int_0^1 \widehat{b}^{()} d \xi_d$$

We now discuss the technical details of the low-rank approximation \eqref{eq:low_rank_f} 
considered in this work. 
 
Assuming that $\mathcal{G}_i$ is given by \eqref{eq:mapG}, we observe that
$$ \widehat{f}_k^{(i)} = \left\lbrace \begin{array}{ll} 2 g^{(i_i)}_k \left( \xi_1,\xi_2,2 \xi_3 \right)  & \text{if } 0 \leq \xi_3 \leq \frac{1}{2}, \\ 2 g^{(i_2)}_k \left( \xi_1,\xi_2,2\xi_3-1 \right) & \text{if } \frac{1}{2} < \xi_3 \leq 1. \end{array} \right. $$
where $g^{(m)}_k = \det \left( J_{\mathcal{F}_m} \right) \left( f_k \circ \mathcal{F}_{m} \right) $ for $m=i_1,i_2$. Note in particular that $\widehat{f}_k^{(i)}$ could be discontinuous at $\xi_3 = \frac{1}{2}$.
%
%\begin{equation} \left\lbrace \begin{array}{ll} \mathcal{F}_{i_1}\left( \xi_1,\xi_2,2 \xi_3 \right)  & \text{if } 0 \leq \xi_3 \leq \frac{1}{2}, \\ \mathcal{F}_{i_2}\left( \xi_1,\xi_2,2\xi_3-1 \right) & \text{if } \frac{1}{2} < \xi_3 \leq 1. \end{array} \right.  \end{equation}
%
We use the {\tt Chebfun3F} algorithm \cite{Dolgov2021} to independently compute low-rank approximations of $g^{(i_1)}_k $ and $g^{(i_2)}_k $, %which we denote with $\widetilde{g}^{(i_1)}_k $ and $\widetilde{g}^{(i_2)}_k $.
i.e.
\begin{equation} \label{eq:g_tilde}
\widetilde{g}^{(m)}_k \approx  \det \left( J_{\mathcal{F}_m} \right) \left( f_k \circ \mathcal{F}_{m} \right)
\end{equation}
for $m=i_1,i_2$. We then consider
\begin{equation} \label{eq:ftilde_glue} \widetilde{f}_k^{(i)} = \left\lbrace \begin{array}{ll} 2 \widetilde{g}^{(i_i)}_k \left( \xi_1,\xi_2,2 \xi_3 \right)  & \text{if } 0 \leq \xi_3 \leq \frac{1}{2}, \\ 2 \widetilde{g}^{(i_2)}_k \left( \xi_1,\xi_2,2\xi_3-1 \right) & \text{if } \frac{1}{2} < \xi_3 \leq 1. \end{array} \right. \end{equation}
Note that $\widetilde{f}_k^{(i)}$ can be easily written as a low-rank function whose multilinear rank is the sum of the multilinear ranks of $g^{(i_1)}_k $ and $g^{(i_2)}_k $. Therefore, as final step we reduce the rank of $\widetilde{f}_k^{(i)}$ using a purely algebraic approach based on the HOSVD.
This truncation is performed with a relative tolerance close to machine precision, in order to remove any redundant terms while at the same time guaranteeing that\eqref{eq:ftilde_glue} still holds. In particular, as it will be useful in the following, we observe that \eqref{eq:ftilde_glue} yields
\begin{equation} \label{eq:ftilde_block}
\vect{v}^T \widetilde{f}^{(i,k)} = \int_{[0,1]^3} \widetilde{f}_{k}^{(i)} (\underline{\xi}) \; \widehat{v}_{sub}^{(i)} (\underline{\xi}) \; \d \underline{\xi} =  \int_{[0,1]^3} \widetilde{g}_k^{(i_1)}(\underline{\xi}) \; \widehat{v}_{ptc}^{(i_1)} (\underline{\xi}) \; + \; \widetilde{g}_k^{(i_2)}(\underline{\xi}) \; \widehat{v}_{ptc}^{(i_2)} (\underline{\xi}) \; \d \underline{\xi},
\end{equation}
where $\widehat{v}_{ptc}^{(m)} = v_h \circ \mathcal{F}_m $, for $m=i_1,i_2$.

%first use the {\tt Chebfun3F} algorithm \cite{Dolgov2021} to independently compute low-rank approximations of $\widehat{f}_{k}^{(i)} = \det \left( J_{\mathcal{G}_i}\right) \left( f_k \circ \mathcal{G}_{i}^{-1} \right)$ for $0 \leq \xi_3 \leq \frac{1}{2}$ and for $\frac{1}{2} < \xi_3 \leq 1$. The reason for this separation is that $J_{\mathcal{G}_i}$ is possibly discontinuous at $\xi_3 = \frac{1}{2}$ (corresponding to the interface between the two patches in the physical domain) and this could hinder the effectiveness of the considered low-rank Chebyshev approximation. 

%Note that also $f_k$ could be discontinuous. We extend the two approximations of $\widehat{f}_{k}^{(i)}$, denoted with $\widetilde{f}_{k,1}^{(i)}$ and $\widetilde{f}_{k,2}^{(i)}$, by zero outside of their respective domain. Then we consider
%$$ \widetilde{f}_{k}^{(i)} = \widetilde{f}_{k,1}^{(i)} + \widetilde{f}_{k,2}^{(i)}. $$
%Note that the multilinear rank of $\widetilde{f}_{k}^{(i)}$ is given by the sum of the multilinear ranks of $\widetilde{f}_{k,1}^{(i)}$ and $\widetilde{f}_{k,2}^{(i)}$. As a final step, we reduce the rank of the approximation using a purely algebraic approach based on the HOSVD.This truncation should be done with a relative tolerance close to machine precision, in order to remove possible redundant terms while at the same time guaranteeing that ?? is still satisfied ()

We finally briefly discuss the Tucker approximation of a given matrix block $ \matb{A}^{(i,i,k,\ell)} $, where $i \in \left\lbrace 1,\ldots,\mathcal{N}_{sub}\right\rbrace $, $k,\ell \in \left\lbrace 1,2,3\right\rbrace $. Let $v_h,w_h \in V^{(i)}_{sub}$, and let $\vect{v}$ and $\vect{w}$ be their vector representations. It holds
\begin{align} \begin{split} \label{eq:Ablock_diag} \vect{v}^T \matb{A}^{(i,i,k,\ell)} \vect{w} & = \int_{\Theta^{(i)}}  2 \mu  \; \varepsilon(\underline{e}_{k} v_h) : \varepsilon(\underline{e}_{\ell} w_h) + \lambda  \left( \nabla \cdot \underline{e}_k v_h \right) \left(\nabla \cdot \underline{e}_{\ell} w_h \right) \d{\underline{x}} = \\ & = \int_{[0,1]^3} \left( \nabla \widehat{v}_{sub}^{(i)} \right)^T Q^{(i,k,\ell)} \nabla \widehat{w}_{sub}^{(i)} \d{\underline{\xi}} ,\end{split} \end{align}
where $\widehat{v}_{sub}^{(i)} = v_h \circ \mathcal{G}_i$, $\widehat{w}_{sub}^{(i)} = v_h \circ \mathcal{G}_i \in \widehat{\boldsymbol{\S}}^{(i)}_{sub,0}$, and
$$ Q^{(i,k,\ell)} = \vert \det \left( J_{\mathcal{G}_i} \right) \vert J_{\mathcal{G}_i}^{-1} \left[  \mu \left( \delta_{k,\ell} I_3 + \underline{e}_{\ell} \underline{e}_k^T \right) + \lambda  
\underline{e}_{k} \underline{e}_{\ell}^T
\right] J_{\mathcal{G}_i}^{-T}.$$
%We consider the following low-rank approximations of the entries of $Q^{(i,k,\ell)}$:
%\begin{equation} \label{eq:low_rank_Q} Q^{(i,k,\ell)}_{\alpha,\beta} \left( \xi_1,\xi_2,\xi_3 \right) \approx \sum_{r_1,r_2,r_3} \mathfrak{q}_{r_1,r_2,r_3}^{(i,k,\ell,\alpha,\beta)} \cdot q^{(i,k,\ell,\alpha,\beta)}_{(1,r_1)} \left( \xi_1 \right) \cdot  q^{(i,k,\ell,\alpha,\beta)}_{(2,r_2)} \left( \xi_2 \right) \cdot  q^{(i,k,\ell,\alpha,\beta)}_{(3,r_3)} \left( \xi_3 \right). \end{equation}
%This approximation can be computed using the same strategy as for \eqref{eq:low_rank_f}.
 
The low-rank approximations of the entries of $Q^{(i,k,\ell)}$ are computed %independently for $0 \leq \xi_3 \leq \frac{1}{2}$ and for $\frac{1}{2} < \xi_3 \leq 1$, 
similarly as for \eqref{eq:low_rank_f}.
We assume as before that $\Theta^{(i)} = \Omega^{(i_1)} \cup \Omega^{(i_2)}$. For each $\alpha,\beta = 1,2,3,$ we use again the approximation \eqref{eq:low_rank_Ablock} of the coefficients $C_{\alpha,\beta}^{i_1,k,\ell}$ and $C_{\alpha,\beta}^{i_2,k,\ell}$.  The two approximations are combined as in \eqref{eq:ftilde_glue} and the corresponding low-rank representation is truncated with a tolerance close to machine precision to remove any redundancy. In the end, the approximations of the entries  of $Q^{(i,k,\ell)}$ are then plugged into \eqref{eq:Ablock_diag}, and thanks to the tensor structure of  $\widehat{\boldsymbol{\S}}_{sub,0}^{(i)}$, we obtain the Tucker approximation \eqref{eq:Ablock_Tucker}. 
As it will be useful in the following, we observe that this approach yields
\begin{equation} \label{eq:Atilde_block2} \vect{v}^T \widetilde{\matb{A}}^{(i,i,k,\ell)} \vect{w} = \int_{[0,1]^3} \left( \nabla \widehat{v}_{ptc}^{(i_i)} \right)^T \widetilde{C}^{(i_1,k,\ell)} \; \nabla \widehat{w}_{ptc}^{(i_1)} \; + \; \left( \nabla \widehat{v}_{ptc}^{(i_2)} \right)^T \widetilde{C}^{(i_2,k,\ell)} \; \nabla \widehat{w}_{ptc}^{(i_2)} \; \d{\underline{\xi}},
 \end{equation}
where $\widehat{v}_{ptc}^{(m)} = v_h \circ \mathcal{F}_m $ and $\widehat{w}_{ptc}^{(m)} = w_h \circ \mathcal{F}_m $, for $m=i_1,i_2$.

%Plugging \eqref{eq:low_rank_Q} into \eqref{eq:Ablock_diag}, and exploiting 
%$$ vect{v}^T \matb{A}^{(i,i,\ell,k)} \vect{w} \approx  \sum_{r_1,r_2,r_3} \mathfrak{q}_{r_1,r_2,r_3}^{(i,k,\ell,\alpha,\beta)} \int_{[0,1]^3} q^{(i,k,\ell,\alpha,\beta)}_{1,r_1} \left( \xi_1 \right) \;  q^{(i,k,\ell,\alpha,\beta)}_{2,r_2} \left( \xi_2 \right) \;  q^{(i,k,\ell,\alpha,\beta)}_{3,r_3} \left( \xi_3 \right) \; \partial_{\alpha} \widehat{v}_{sub}^{(i)} \; \partial_{\beta} \widehat{w}_{sub}^{(i)} \; d \underline{\xi} $$
%the tensor structure of  $\widehat{\boldsymbol{\S}}_{sub,0}^{(i)}$, we obtain the Tucker approximation \eqref{eq:Ablock_Tucker}. 

\begin{rmk}
Since the
low-rank strategy is meant to  effectively tackle large problems,  we
have considered above the subdomains 
as union of adjacent  patches. %Other strategies are possible, in particular one could consider a larger number of smaller subdomains.
It would be of course possible to consider a larger number of smaller subdomains.
We emphasize, however, that a large number of
subdomains would require the introduction of a coarse solver. Indeed, it is well-known that
preconditioners based on local domain decomposition without coarse
solver are not robust with respect to the number of
subdomains. 
\end{rmk}

\section{Analysis of the low-rank approximation of the linear system}
\label{sec:analysis}

We now show some theoretical results related to the low-rank approximation of the linear system \eqref{eq:low_rank_sys}. We adopt the following notation: given 
$\underline{v}_h \in \left[ V_h \right]^{3} $, we denote with 
$\nabla \underline{v}_h$ the vector of $\mathbb{R}^9$ obtained by stacking the gradients of the components of $\underline{v}_h$.

%We consider $\underline{v}_h,\underline{w}_h \in left[ V_h \right]^3 $ and we write $\underline{v}_h = \sum_{i=1}^{\mathcal{N}_{sub}} \underline{v}_h^{(i)}$ and $\underline{w}_h = \sum_{i=1}^{\mathcal{N}_{sub}} \underline{w}_h^{(i)}$ with $v_h^{(i)},w_h^{(i)} \in \left[ V_{sub}^{(i)} \right]^3 $. Let $\vect{v}^{(i)}$ and $\vect{w}^{(i)}$ denote the vector representations of $v_h^{(i)}$ and $w_h^{(i)}$, respectively, and let .

%Given $ \vect{v}= \displaystyle \begin{bmatrix} \vect{v}^{(1)} \\ \vdots \\ \vect{v}^{(\mathcal{N}_{sub})} \end{bmatrix} $ and $ \vect{w}= \displaystyle \begin{bmatrix} \vect{w}^{(1)} \\ \vdots \\ \vect{w}^{(\mathcal{N}_{sub})} \end{bmatrix} $, we consider $\underline{v}_h = \sum_{i=1}^{\mathcal{N}_{sub}} \underline{v}_h^{(i)}$ and $\underline{w}_h = \sum_{i=1}^{\mathcal{N}_{sub}} \underline{w}_h^{(i)}$, where 

Let $\underline{v}_h,\underline{w}_h \in \left[ V_h \right]^3 $, and let $\vect{v},\vect{w}$ be block vectors representing these functions with respect to $\left[  V_h^{(1)} \right]^3,\ldots,$\\
$\left[ V_h^{(\mathcal{N}_{sub})} \right]^3$ and their respective bases. We preliminary observe that
\begin{equation} \label{eq:A_bilinearform} \vect{v}^T \matb{A} \vect{w} = \sum_{m=1}^{\mathcal{N}_{ptc}} \int_{[0,1]^3} \left( \nabla \widehat{\underline{v}}_{ptc}^{(m)} \right)^T C^{(m)} \; \nabla \widehat{\underline{w}}_{ptc}^{(m)} \; \d{\underline{\xi}}, \end{equation}
%where $\underline{v}_h = \sum_{i=1}^{\mathcal{N}_{sub}} \underline{v}_h^{(i)}$ and $\underline{w}_h = \sum_{i=1}^{\mathcal{N}_{sub}} \underline{w}_h^{(i)}$
where $ \widehat{\underline{v}}_{ptc}^{(m)} = \underline{v}_h \circ \mathcal{F}_{m}$, $\widehat{\underline{w}}_{ptc}^{(m)} = \underline{w}_h \circ \mathcal{F}_{m}$, and 
\begin{equation} \label{eq:Cm} C^{(m)}(\underline{\xi}) = \begin{bmatrix} C^{(m,1,1)}(\underline{\xi}) & C^{(m,1,2)}(\underline{\xi}) & C^{(m,1,3)}(\underline{\xi}) \\ C^{(m,2,1)}(\underline{\xi}) & C^{(m,2,2)}(\underline{\xi}) & C^{(m,2,3)}(\underline{\xi}) \\ C^{(m,3,1)}(\underline{\xi}) & C^{(m,3,2)}(\underline{\xi}) & C^{(m,3,3)}(\underline{\xi})  \end{bmatrix} \in \mathbb{R}^{9 \times 9}, \end{equation}
whose blocks are defined in \eqref{eq:Cmkl}.

In the next lemma we show in particular that $\widetilde{\mathbf{A}}$ defines a bilinear form analagous to \eqref{eq:A_bilinearform}.

\begin{lemma} \label{lemma}

Let $\underline{v}_h,\underline{w}_h \in \left[ V_h \right]^3 $, and let $\vect{v},\vect{w}$ be block vectors representing these functions with respect to $\left[  V_h^{(1)} \right]^3,\ldots,\left[ V_h^{(\mathcal{N}_{sub})} \right]^3$ and their respective bases. Then
\begin{align} \label{eq:Atilde_bilinearform} \vect{v}^T \widetilde{\mathbf{A}} \vect{w} & =  \sum_{m=1}^{\mathcal{N}_{ptc}} \int_{[0,1]^3} \left( \nabla \widehat{\underline{v}}_{ptc}^{(m)} \right)^T \widetilde{C}^{(m)} \; \nabla \widehat{\underline{w}}_{ptc}^{(m)} \; \d{\underline{\xi}}, \\ 
\label{eq:ftilde_functional} \vect{v}^T \widetilde{\vect{f}} & = \sum_{m=1}^{\mathcal{N}_{ptc}} \int_{[0,1]^3}  \left(  \widehat{\underline{v}}_{ptc}^{(m)} \right)^T \underline{\widetilde{g}}^{(m)} \; \d{\underline{\xi}},
\end{align}
where $ \widehat{\underline{v}}_{ptc}^{(m)} = \underline{v}_h \circ \mathcal{F}_{m}$, $\widehat{\underline{w}}_{ptc}^{(m)} = \underline{w}_h \circ \mathcal{F}_{m}$, $\widetilde{C}^{(m)}(\underline{\xi})$ is the $3 \times 3$ block matrix whose blocks are the $3 \times 3$ matrices $\widetilde{C}^{(m,k,\ell)}(\underline{\xi})$ from \eqref{eq:low_rank_C} for $k,\ell=1,2,3$, while $\underline{\widetilde{g}}^{(m)}(\underline{\xi}) =  \begin{bmatrix} \widetilde{g}_1^{(m)}(\underline{\xi}) & \widetilde{g}_2^{(m)}(\underline{\xi}) & \widetilde{g}_3^{(m)}(\underline{\xi}) \end{bmatrix}^T $, whose entries are introduced in \eqref{eq:g_tilde}. 

\end{lemma}

\begin{proof}

By considering \eqref{eq:Atilde_block1} and \eqref{eq:Atilde_block2} for $k,\ell=1,2,3,$ we infer
$$ \vect{v}^T \widetilde{\matb{A}} \vect{w} 
= \sum_{i,j=1}^{\mathcal{N}_{sub}}
\left( \vect{v}^{(i)} \right)^T \widetilde{\matb{A}}^{(i,j)} \vect{w}^{(j)} = \sum_{i,j=1}^{\mathcal{N}_{sub}} \; \sum_{m=1}^{\mathcal{N}_{ptc}} \int_{[0,1]^3} \left( \nabla \widehat{\underline{v}}_{ptc}^{(m,i)} \right)^T \widetilde{C}^{(m)} \nabla \widehat{\underline{w}}_{ptc}^{(m,j)} \; \d{\underline{\xi}} ,$$
where $\widehat{\underline{v}}_{ptc}^{(m,i)} = \underline{v}_h^{(i)} \circ \mathcal{F}_m $ and $\widehat{\underline{v}}_{ptc}^{(m,j)} = \underline{v}_h^{(j)} \circ \mathcal{F}_m $. Note that at least one between $\widehat{\underline{v}}_{ptc}^{(m,i)}$ and $\widehat{\underline{w}}_{ptc}^{(m,j)}$  is the null function unless $\Omega^{(m)} \subseteq \Theta^{(i)}\cap \Theta^{(j)}$. %, making all terms in the above sum zero except one if $i \neq j$ or two if $i = j$. We plug this expression into ?? and switch the order of the sums, obtaining
By switching the order of the sums, we obtain
$$ \sum_{m=1}^{\mathcal{N}_{ptc}} \sum_{i,j=1}^{\mathcal{N}_{sub}}
\int_{[0,1]^3} \left( \nabla \widehat{\underline{v}}_{ptc}^{(m,i)} \right)^T \widetilde{C}^{(m)} \nabla \widehat{\underline{w}}_{ptc}^{(m,j)} \; \d{\underline{\xi}} = \sum_{m=1}^{\mathcal{N}_{ptc}} \int_{[0,1]^3} \left( \sum_{i=1}^{\mathcal{N}_{sub}} \nabla \widehat{\underline{v}}_{ptc}^{(m,i)} \right)^T \widetilde{C}^{(m)} \left( \sum_{j=1}^{\mathcal{N}_{sub}} \nabla \widehat{\underline{w}}_{ptc}^{(m,j)} \right) \; \d{\underline{\xi}}. $$
Then \eqref{eq:Atilde_bilinearform} follows by observing that $\sum_{i=1}^{\mathcal{N}_{sub}} \widehat{\underline{v}}_{ptc}^{(m,i)} = \widehat{\underline{v}}_{ptc}^{(m)} $ and $\sum_{j=1}^{\mathcal{N}_{sub}} \widehat{\underline{w}}_{ptc}^{(m,j)} = \widehat{\underline{w}}_{ptc}^{(m)} $. 
Similarly, considering \eqref{eq:ftilde_block} for $k= 1,2,3$, we infer
$$  \vect{v}^T \widetilde{\vect{f}} = \sum_{i=1}^{\mathcal{N}_{sub}} \left( \vect{v}^{(i)} \right)^T \widetilde{\vect{f}}^{(i)} = \sum_{i=1}^{\mathcal{N}_{sub}} \sum_{i=1}^{\mathcal{N}_{ptc}} \int_{[0,1]^3} \left( \widehat{\underline{v}}_{ptc}^{(m,i)} \right)^T \underline{\widetilde{g}}^{(m)} \; \d{\underline{\xi}} = \sum_{i=1}^{\mathcal{N}_{ptc}} \int_{[0,1]^3} \left( \sum_{i=1}^{\mathcal{N}_{sub}} \widehat{\underline{v}}_{ptc}^{(m,i)} \right)^T \underline{\widetilde{g}}^{(m)} \; \d{\underline{\xi}}  $$
which proves \eqref{eq:ftilde_functional}.
%Identity {eq:ftilde_functional} can be proven analogously by summing 
\end{proof}

The following result shows that the low-rank  approximation described in Section \ref{sec:ptc_union} does not alter the kernel of the system matrix, and that problem \eqref{eq:low_rank_sys}  admits a solution. This result is based on the assumption that the matrix  $\widetilde{C}^{(m)}(\underline{\xi})$ introduced in Lemma \ref{lemma} is positive definite everywhere, which is clearly satisfied provided that the low-rank approximations \eqref{eq:low_rank_C} are accurate enough.

\begin{thm} Assume that the matrix $\widetilde{C}^{(m)}(\underline{\xi})$ defined in Lemma \ref{lemma} is positive definite for every $m = 1,\ldots, \mathcal{N}_{ptc}$ and for every $\underline{\xi} \in [0,1]^3$. Then it holds
$$ \ker \left( \widetilde{\matb{A}} \right) = \ker \left( \matb{A} \right) \qquad \text{and} \qquad \widetilde{\vect{f}} \in \mathrm{range} \left( \widetilde{\matb{A}} \right). $$
\end{thm}

\begin{proof}

Let $\vect{v}$ be a block vector representing the function $\underline{v}_h \in \left[ V_h \right]^3$. %We now show that $\vect{v}^T \widetilde{A} \vect{v} = 0$ if and only if $\underline{v}_h$ is the null function, implying that $\ker \left( \widetilde{\matb{A}} \right) = \ker \left( \matb{A} \right)$. 
Thanks to Lemma \ref{lemma}, we have
$$ \vect{v}^T \widetilde{\matb{A}} \vect{v} = \sum_{m=1}^{\mathcal{N}_{ptc}} \int_{[0,1]^3} \left( \nabla \widehat{\underline{v}}_{ptc}^{(m)} \right)^T \widetilde{C}^{(m)} \nabla \widehat{\underline{v}}_{ptc}^{(m)} \; \d{\underline{\xi}} \geq \min_{m=1,\ldots,\mathcal{N}_{ptc}} \; \min_{\underline{\xi} \in [0,1]^3}  \lambda_{\min} \left( \widetilde{C}^{(m)} (\underline{\xi}) \right) \sum_{m=1}^{\mathcal{N}_{ptc}} \int_{[0,1]^3} \Vert \nabla \widehat{\underline{v}}_{ptc}^{(m)} \Vert_2^2 \; \d{\underline{\xi}}.$$ 
The latter term is zero if and only if $\widehat{\underline{v}}_{ptc}^{(m)}$ is the null function for every $m = 1,\ldots,\mathcal{N}_{ptc}$, which is equivalent to require that $\underline{v}_h$ is the null function. Since $\ker \left( \matb{A} \right)$ is the set of vectors representing the null function, we have that $\ker \left( \matb{A} \right) = \ker \left( \widetilde{\matb{A}} \right) $. Then,
from \eqref{eq:ftilde_functional} we infer that $\vect{v}^T \widetilde{\vect{f}} = 0 $ whenever $\vect{v}$ is a block vector representing the null function. Therefore, $\widetilde{\vect{f}} \in \ker \left( \widetilde{\matb{A}}\right)^\perp = \text{range} \left( \widetilde{\matb{A}} \right) $.

%Then, to show that $\widetilde{\vect{f}} \in \text{range} \left( \widetilde{\matb{A}} \right)$, we just need to prove that $\vect{v}^T \widetilde{\vect{f}} = 0 $ whenever $\vect{v}$ is a block vector representing the null function. But this is true

\end{proof}

We are now ready to prove an upper bound on the  relative error in the 2-norm committed by approximating $\matb{A}$ with $\widetilde{\matb{A}}$, in terms of the accuracy of the low-rank approximations \eqref{eq:low_rank_C}.

\begin{thm}
Assume that, for every $\alpha,\beta,k,\ell = 1,2,3$ and for every $m = 1,\ldots, \mathcal{N}_{ptc}$, the low-rank approximations \eqref{eq:low_rank_C} satify
\begin{equation} \label{eq:matrix_error_bound} \max_{\underline{\xi} \in [0,1]^3}\vert C_{\alpha,\beta}^{(m,k,\ell)}(\underline{\xi}) - \widetilde{C}_{\alpha,\beta}^{(m,k,\ell)}(\underline{\xi}) \vert \leq \zeta \end{equation}
for some  $\zeta\geq 0$. Then 
$$ \frac{\Vert \matb{A} - \widetilde{\matb{A}} \Vert_2}{\Vert \matb{A} \Vert_2} \leq \frac{9 \;  \zeta}{ \displaystyle \min_{m=1,\ldots,\mathcal{N}_{ptc}} \; \min_{\underline{\xi} \in [0,1]^3} \lambda_{\min} \left( C^{(m)}\left( \underline{\xi}\right) \right) },$$
where $\lambda_{\min} \left( C^{(m)}\left( \underline{\xi} \right) \right)$ is the minimum eigenvalue of the matrix $C^{(m)}\left(\underline{\xi}\right) \in \mathbb{R}^{9 \times 9}$ defined in \eqref{eq:Cm}.
\end{thm}

\begin{proof}
Since $\matb{A}$ and $\widetilde{\matb{A}}$ are symmetric, it holds
$$ \Vert \matb{A} - \widetilde{\matb{A}} \Vert_2 = \max_{\vect{v}} \frac{\vect{v}^T \left( \matb{A} - \widetilde{\matb{A}} \right) \vect{v}}{ \Vert \vect{v} \Vert_2^2}.$$
Exploiting \eqref{eq:Atilde_block1} and \eqref{eq:Atilde_block2}, we infer 
\begin{align*} \vect{v}^T \left( \matb{A} - \widetilde{\matb{A}} \right) \vect{v} & =  \sum_{m=1}^{\mathcal{N}_{ptc}}
\int_{[0,1]^3} \left( \nabla \widehat{\underline{v}}_{ptc}^{(m)} \right)^T \left( C^{(m)} - \widetilde{C}^{(m)}\right) \nabla \widehat{\underline{v}}_{ptc}^{(m)} \; \d{\underline{\xi}} \\
& \leq \max_{m = 1,\ldots,\mathcal{N}_{ptc}} \; \max_{\underline{\xi} \in [0,1]^3} \lambda_{\max} \left( C^{(m)}(\underline{\xi}) - \widetilde{C}^{(m)}(\underline{\xi}) \right) \; \sum_{m=1}^{\mathcal{N}_{ptc}}
\int_{[0,1]^3} \Vert \nabla \widehat{\underline{v}}_{ptc}^{(m)} \Vert^2_2 \; \d{\underline{\xi}} \\
& \leq \frac{\displaystyle \max_{m = 1,\ldots,\mathcal{N}_{ptc}} \; \max_{\underline{\xi} \in [0,1]^3} \lambda_{\max} \left( C^{(m)}(\underline{\xi}) - \widetilde{C}^{(m)}(\underline{\xi}) \right)}{\displaystyle \min_{m=1,\ldots,\mathcal{N}_{ptc}} \; \min_{\underline{\xi} \in [0,1]^3} \lambda_{\min} \left( C^{(m)}\left( \underline{\xi}\right) \right)} \sum_{m=1}^{\mathcal{N}_{ptc}}
\int_{[0,1]^3} \left( \nabla \widehat{\underline{v}}_{ptc}^{(m)} \right)^T  C^{(m)} \nabla \widehat{\underline{v}}_{ptc}^{(m)} \; \d{\underline{\xi}} \\
& = \frac{\displaystyle \max_{m = 1,\ldots,\mathcal{N}_{ptc}} \; \max_{\underline{\xi} \in [0,1]^3} \lambda_{\max} \left( C^{(m)}(\underline{\xi}) - \widetilde{C}^{(m)}(\underline{\xi}) \right)}{\displaystyle \min_{m=1,\ldots,\mathcal{N}_{ptc}} \; \min_{\underline{\xi} \in [0,1]^3} \lambda_{\min} \left( C^{(m)}\left( \underline{\xi}\right) \right)} \;\vect{v}^T \matb{A} \vect{v}.
 \end{align*}
Inequality \eqref{eq:matrix_error_bound} then follows by observing that $ \displaystyle \max_{\vect{v}} \frac{\vect{v}^T \matb{A} \vect{v}}{\Vert \vect{v} \Vert_2^2} = \Vert \matb{A} \Vert_2 $ and that
$$ \lambda_{\max} \left( C^{(m)} (\underline{\xi})  - \widetilde{C}^{(m)}(\underline{\xi}) \right) \leq \Vert  C^{(m)}(\underline{\xi}) - \widetilde{C}^{(m)}(\underline{\xi}) \Vert_{\infty} \leq 9 \; \max_{\alpha,\beta,k,\ell = 1,2,3} \vert C_{\alpha, \beta}^{(m,k,\ell)}(\underline{\xi}) - \widetilde{C}_{\alpha, \beta}^{(m,k,\ell)}(\underline{\xi}) \vert \leq 9 \; \zeta $$
for every $m=1,\ldots, \mathcal{N}_{ptc}$ and for every $\underline{\xi} \in [0,1]^3$.

\end{proof}

\section{Numerical experiments}
\label{sec:numerics} 
In this section we present some numerical experiments to show the behavior of our low-rank solving strategy for multipatch geometries. 
All the tests are performed with 
Matlab R2022b on an Intel(R) Core(TM) i9-12900K, running at  3.20 GHz, and with 128 GB of RAM.
%All the tests are performed with Matlab R2021a on an 16-Core Intel Xeon W, running at 3.20 GHz, and with 384 GB of RAM. 

We used GeoPDEs toolbox \cite{Vazquez2016} to handle isogeometric discretizations. The approximation of the right-hand side and the linear system matrix are computed using Chebfun toolbox functions \cite{Driscoll2014}.

In all the tests, we employ a dyadic refinement and every patch has
the same number of elements $n_{el}=2^{L}$ in each parametric
direction and for different discretization levels $L$. The numerical
solution is computed with the TPCG method, discussed in Section
\ref{sec:tpcg}, with relative tolerance $tol=10^{-6}$ and the zero
vector as initial guess. In all the numerical tests we set  the right-hand side equal to $\underline{f}=[0, \ 0, \ -1]^T$. The  system matrix and the right-hand side
are approximated with the technique of Section \ref{sec:ptc_union} by
setting the relative tolerance of {\tt Chebfun3F} equal to $10^{-1} tol = 10^{-7}$ (see \cite{Dolgov2021} for more details).
The dynamic truncation described in Section \ref{sec:tpcg} is used fixing the initial relative tolerance as $\epsilon_0=10^{-1}$ and the minimum relative tolerance as $\epsilon_{\min}=tol\|\vect{f}\|_2 10^{-1}$. The threshold $\delta$ is set equal to $10^{-3}$, the reducing factor $\alpha$ is set equal to 0.5. The relative tolerance parameter $\beta$ is chosen as $10^{-1}$, while the intermediate relative tolerance $\gamma$ is chosen as $10^{-8}$. The relative tolerance $\epsilon_{rel}^{prec}$ for the approximation of the diagonal blocks of the preconditioner (see \cite{Montardini2023} for details) is set equal to $10^{-1}$. We experimentally verified that these choices of the parameters yield good performance of the TPCG method.

{
The Lam\'e parameters are defined as
\begin{equation} \label{eq:lame}
\mu = \frac{E}{2 \left( 1 + \nu \right)}, \qquad \lambda = \frac{E \nu}{\left( 1 + \nu \right)\left( 1 - 2\nu \right)},
\end{equation}
where $E$ is the Young's modulus and $\nu$ the Poisson ratio. Throughout this section, unless otherwise specified, we set $E=1$ and $\nu=0.3$.
}

Let  
$$\widetilde{\vect{u}}^{(j)}:= \begin{bmatrix} \widetilde{\vect{u}}^{(j,1)} \\ \widetilde{\vect{u}}^{(j,2)} \\\widetilde{\vect{u}}^{(j,3)} \end{bmatrix}, \qquad j = 1,\ldots,\mathcal{N}_{sub} $$ 
be the solution of the linear elasticity problem in Tucker format in each subdomain. We denote with $(R^{\vect{u}}_1(j,d), R^{\vect{u}}_2(j,d), R^{\vect{u}}_3(j,d))$ the multilinear rank of $\widetilde{\vect{u}}^{(j,d)}$, for $ d=1,2,3$ and $j=1,\dots,\numsub$. { The memory requirements in percentage of the low-rank solution with respect to the full solution
is defined as  
\begin{equation}
\label{eq:mem_comp}
\text{memory requirements in \%}: = \frac{\sum_{j=1}^{\numsub}\sum_{d=1}^3 \left(R^{\vect{u}}_1(j,d)n_{sub,1}^{(j)} + R^{\vect{u}}_2(j,d)n_{sub,2}^{(j)} + R^{\vect{u}}_3(j,d)n_{sub,3}^{(j)}\right)}{N_{dof}}\cdot 100,
\end{equation}}
where $N_{dof}:=\sum_{j=1}^{\numsub} n_{sub,1}^{(j)} n_{sub,2}^{(j)} n_{sub,3}^{(j)}$ represents the total number of degrees of freedom, i.e. the dimension of the space $\sum_{j=1}^{\numsub}V_{sub}^{(j)}$. 
Note that we only consider the memory compression of the solution and not the memory compression of the right-hand side and of the linear system matrix, that has already been analyzed in \cite{Mantzaflaris2017,Hofreither2018,Bunger2020}.

In the numerical experiments we show the maximum of the multilinear rank of all the components of the solution. Thanks to the strategy employed for the truncation of the iterates $\vect{r}_k, \vect{z}_k , \vect{p}_k$ and $\vect{q}_k$, the  maxima   of their multilinear ranks is comparable to the one of $\vect{u}_k$ when convergence approaches, and thus their values are not reported.

 Perhaps not surprisingly, our numerical experience indicates that the main effort of the TPCG iteration in terms of computation time is represented by the truncation steps. Anyway, since our Matlab implementation is not optimized, we prefer not to report a complete distribution of the computation time throughout the algorithm.

\subsection{L-shaped domain}
In this test we consider a three dimensional L-shaped domain, as represented in Figure \ref{fig:in_dom}, where different colors represent different patches. We have two subdomains, one is the union of the yellow and light green patches and the other the union of the light green and dark green patches.
%For this domain, we preliminary perform a convergence test to assess that the numerical solution yielded by our approach is indeed a good approximation of the Galerkin solution. For this purpose, we consider a Poisson toy problem with homogeneous Dirichlet boundary conditions and source function $f(x,y,z)=\sin(2 \pi x) \sin(2 \pi y)  \sin( 2 \pi z)$. Note that the exact solution of this toy problem is $u(x,y,z) = -\left(12 \pi^2 \right)^{-1} f(x,y,z) $.
%We assemble and solve the problem with a low-rank strategy analogous to that described previously (in particular we fix $10^{-10}$ as tolerance for TPCG) and then compute the $H^1$ and $L^2$ norms of the error for the obtained solution, for different discretization levels and polynomial degrees. Results are shown in Figure \ref{fig:L_results_error} and confirm the validity of our approach. 
%\begin{figure}[H]
% \centering  
%     \subfloat[][$L^2$ error.\label{fig:L2}] 
%   {\includegraphics[scale=0.33]{L2}}\quad
% \subfloat[][$H^1$ error.\label{fig:H1}]
%   {\includegraphics[scale=0.33]{H1}}
%   \caption{Errors for a Poisson problem in the  L-shaped domain.}
% \label{fig:L_results_error}
%\end{figure} 
We % now solve a linear elasticity problem on the same domain, we
%  set the right-hand side equal to $\underline{f}=[0, \ 0, \ -0.3]^T$ and we 
  consider  homogeneous Dirichlet boundary conditions on the patch faces that lie on the planes $\{x=-1\}\cup\{z=1\}\cup \{x=0\} \cup \{z=0\} $, and homogeneous Neumann boundary conditions elsewhere. 
 
\begin{figure}[H]
\centering 
   \includegraphics[scale=0.40]{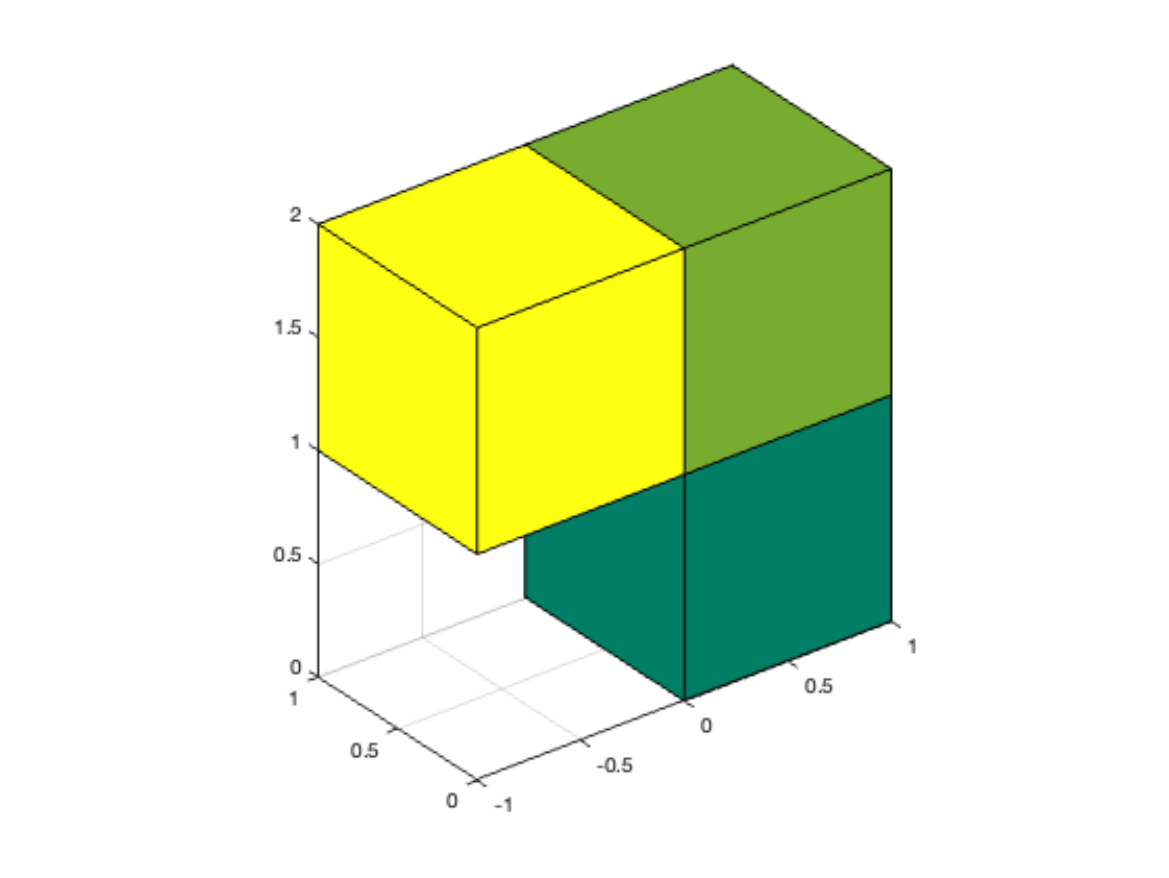}
  \caption{L-shaped domain.}\label{fig:in_dom}
\end{figure}

For this domain, the multilinear ranks of the blocks of $\widetilde{\matb{A}}$ 
(note that here the geometry is trivial for all patches, therefore $\matb{A} = \widetilde{\matb{A}}$) 
are
$$
 \left(R^A_1(i,j,k,k),R^A_2(i,j,k,k),R^A_3(i,j,k,k)\right) = (3,3,3), \qquad k=1,2,3,
$$
and 
$$\left(R^A_1(i,j,k,\ell), R^A_2(i,j,k,\ell),R^A_3(i,j,k,\ell)\right) = (2,2,2), \qquad k,\ell = 1,2,3, \; k \neq \ell, $$
for $i,j = 1,2$.

In Table \ref{tab:its_Lshaped} we report the number of iterations for degrees $p=3,4,5$ and number of elements per patch per parametric direction $n_{el}=2^{L}$ with $L=5,6,7,8.$ The number of iterations  is almost independent of the degree and grows only mildly with the number of elements $n_{el}$. We emphasize that for the finer discretization level the number of degrees of freedom is roughly 200 millions.

 {\R1 \renewcommand\arraystretch{1.4} 
\begin{table}[H]
\begin{center}
%\footnotesize
\begin{tabular}{|c|c|c|c|}
\hline
 & \multicolumn{3}{|c|}{ \ Iteration number} \\
 \hline$n_{el}$  & $p=3$ &$p=4$  & $p=5$ \\
\hline  
32    &   36  &  34   &  35 \\
\hline
64    &   38  &  36  &   37  \\ 
\hline
128   &   42  &   40 &  41  \\ 
\hline
256    &  51   & 48    &  49  \\ 
\hline
\end{tabular}
\caption{Number of iterations for the L-shaped domain.}
\label{tab:its_Lshaped}
\end{center}
\end{table}}

Figure \ref{fig:rank_L} shows the maximum of the  multilinear ranks of the solution in all the displacement directions and in all the subdomains: the ranks grow as the degree and number of elements is increased, implying that the solution has not a low rank. 
The memory compression \eqref{eq:mem_comp}, represented in Figure \ref{fig:mem_L}, highlights the great memory saving that we have with the multipatch low-rank strategy with respect to the full one. 
Note that the maximum rank of the solution is rather large, between 100 and 150 on the finest discretization level. Nevertheless, the memory compression is around $1\%$ for this level.

\begin{figure}[H]
 \centering  
     \subfloat[][Maximum of the multilinear rank of the solution.\label{fig:rank_L}] 
   {\includegraphics[scale=0.43]{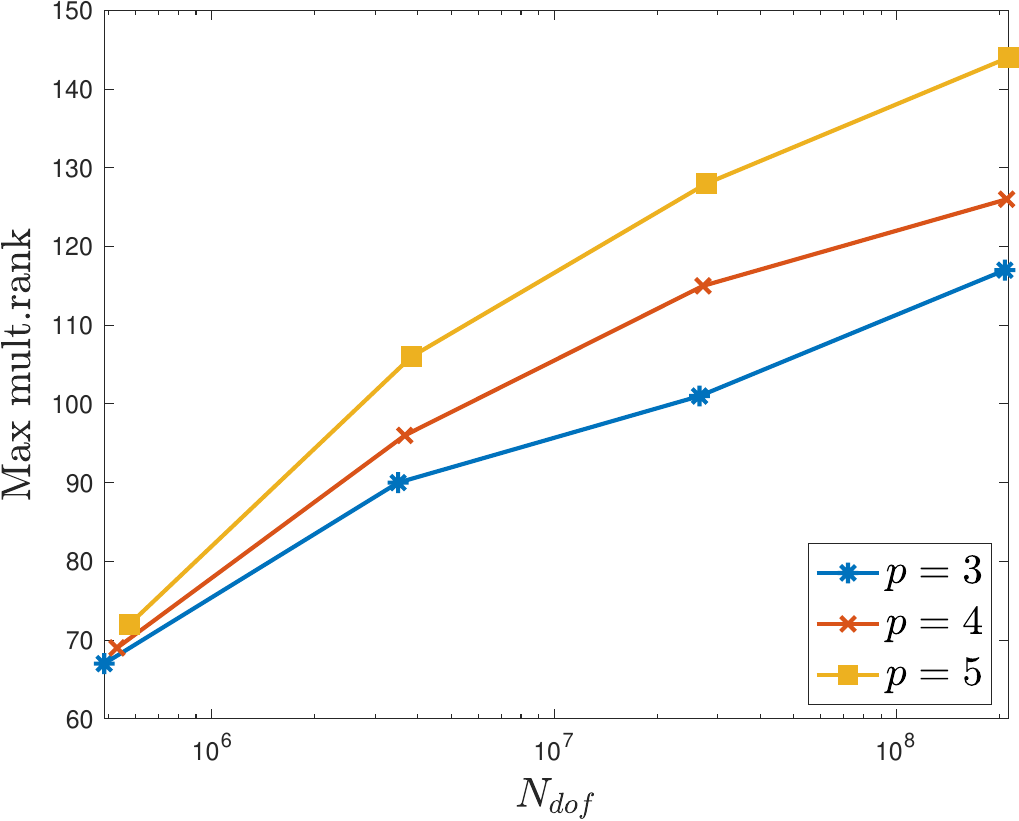}}\quad
 \subfloat[][Memory requirements.\label{fig:mem_L}]
   {\includegraphics[scale=0.43]{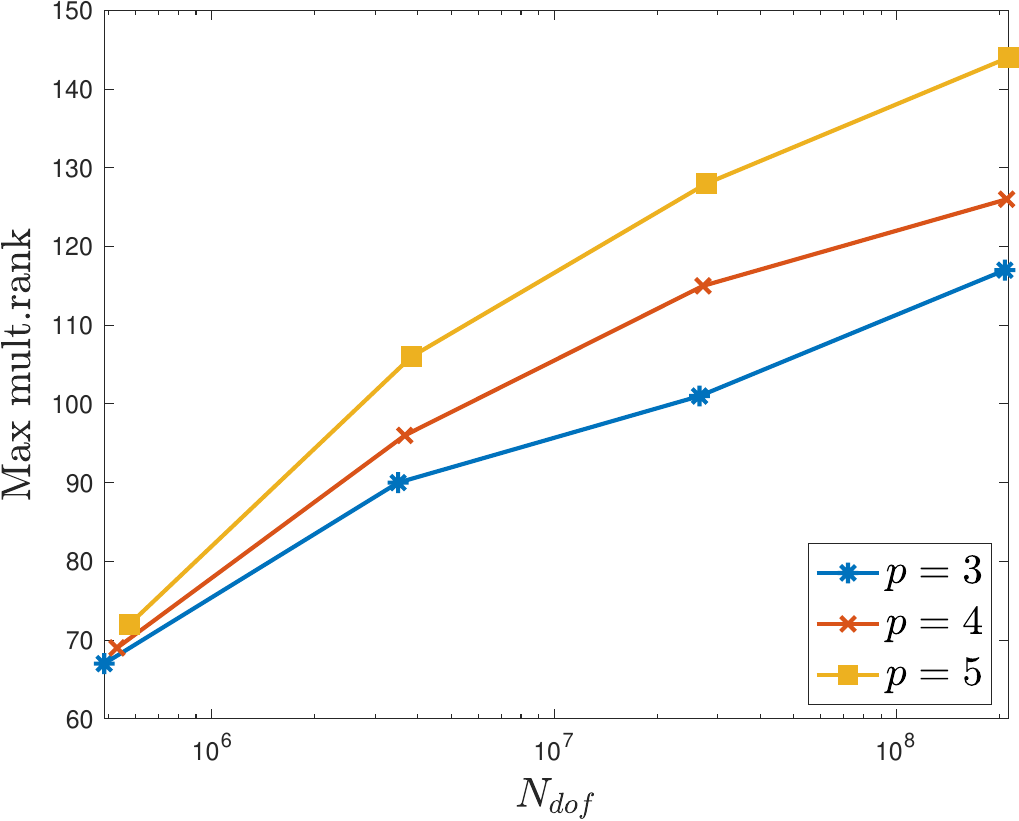}}
   \caption{Results for the L-shaped domain.}
 \label{fig:L_results}
\end{figure}

{
\subsection{3D-cross domain}
In this test, we consider a 3D-cross domain,  as represented in Figure \ref{fig:cross}. Note that in this domain we have that the central cube belongs to six subdomains, which is the maximum allowed by our strategy. We consider homogeneous Dirichlet boundary conditions everywhere, except for the external faces. %, and right-hand side equal to $\underline{f}=[0, \ 0, \ -2]^T$. %The numerical solution is represented in Figure \ref{fig:cross_sol}.
The multilinear ranks of the blocks of $\widetilde{\mathbf{A}}$ are 
$$
 \left(R^A_1(i,j,k,k),R^A_2(i,j,k,k),R^A_3(i,j,k,k)\right) = (3,3,3), \qquad k=1,2,3,
$$
and 
$$\left(R^A_1(i,j,k,\ell), R^A_2(i,j,k,\ell),R^A_3(i,j,k,\ell)\right) = (2,2,2), \qquad k,\ell = 1,2,3, \; k \neq \ell, $$
for $i,j = 1,\ldots,6$. 
}

\begin{figure}[H]
\begin{center}
\includegraphics[scale=0.35]{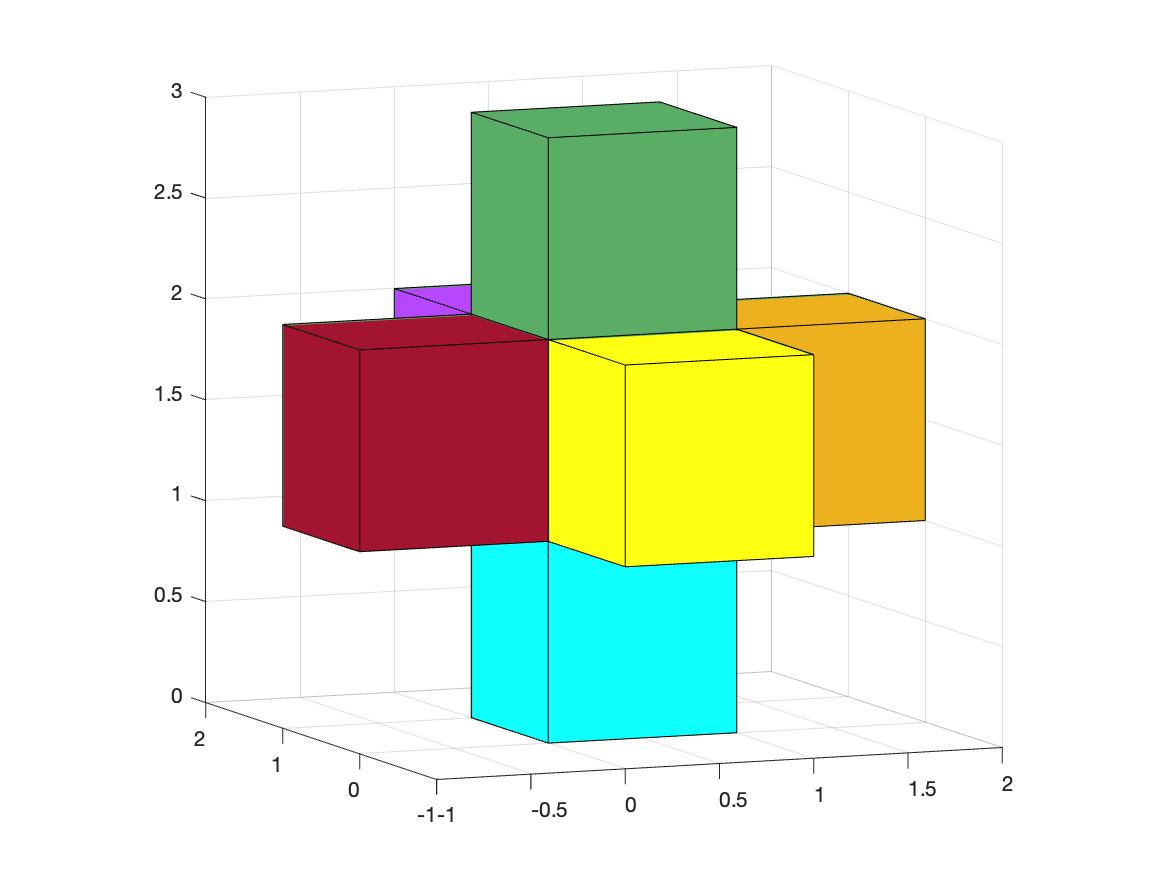}
  \caption{3D-cross domain.}\label{fig:cross}
     \end{center}
\end{figure}

{

The number of iterations for degrees  $p=3,4,5$ and number of elements per parametric direction $n_{el}=2^L$ with $L=5,6,7,8$ is reported in Table \ref{tab:its_Cross}. The number of iterations is almost independent on the degree and mesh-size. These numbers are similar, although slightly higher, to the ones of Table \ref{tab:its_Lshaped} relative to the L-shaped domain: the proposed low-rank strategy seems  to be independent on the number of subdomains to which a patch belongs.
 {\renewcommand\arraystretch{1.4} 
\begin{table}[H]
\begin{center}
%\footnotesize
\begin{tabular}{|c|c|c|c|}
\hline
 & \multicolumn{3}{|c|}{ \ Iteration number} \\
 \hline$n_{el}$ &   $p=3$ &$p=4$  & $p=5$ \\
\hline  
32 &     51  &  51   & 50 \\
\hline
64 &       54  & 52  &  51  \\ 
\hline
128 &      53 &  54  & 53  \\ 
\hline
256 &     57  &  57   &  57 \\ 
\hline
\end{tabular}
\caption{Number of iterations for the cross domain.}
\label{tab:its_Cross}
\end{center}
\end{table}}

We report in Figure \ref{fig:rank_cross} the maximum of the multilinear ranks of the solution in all the directions and for all the subdomains. The multilinear rank of the solution increases with the number of elements per direction $n_{el}$. However,  the memory storage for the solution is hugely reduced with respect to the full solution and reaches values around $1\%$ for the highest discretization level, as represented in Figure \ref{fig:mem_cross}.

\begin{figure}[H]
 \centering  
     \subfloat[][Maximum of the multilinear rank of the solution.\label{fig:rank_cross}] 
   {\includegraphics[scale=0.40]{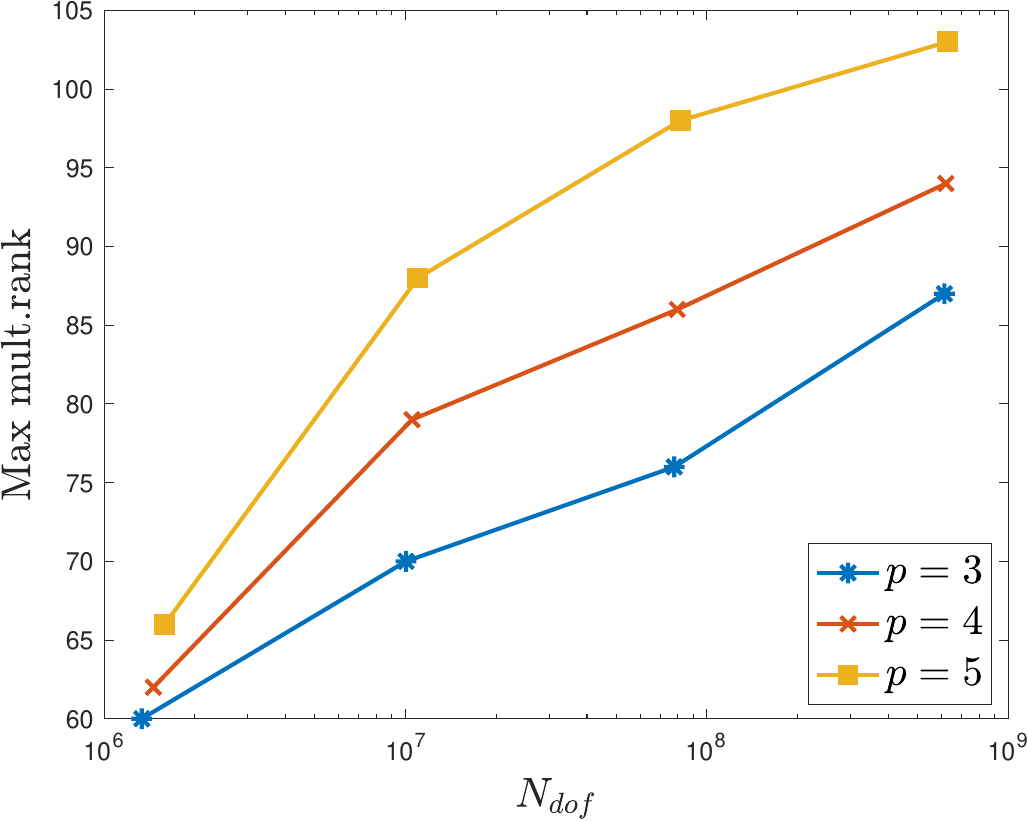}}\quad
 \subfloat[][Memory requirements.\label{fig:mem_cross}]
   {\includegraphics[scale=0.40]{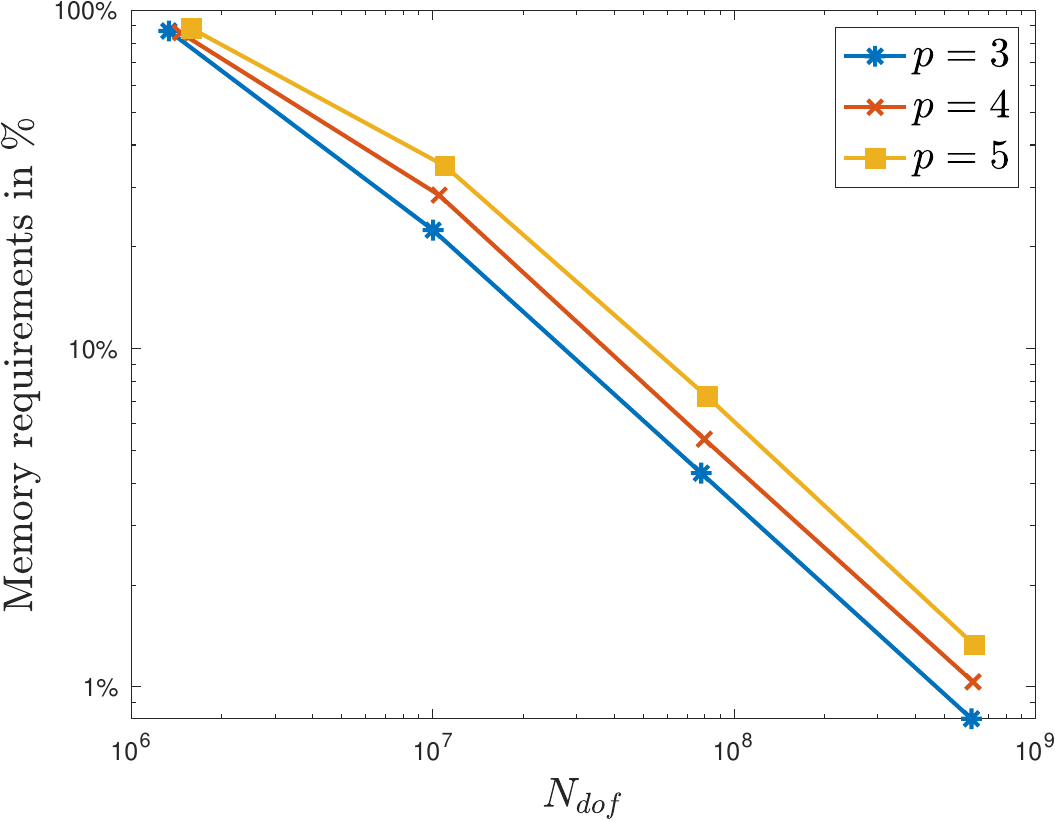}}
   \caption{Results for the 3D-cross domain.}
 \label{fig:cross_results}
\end{figure}}

{\R1
\subsection{Thick square with jumping coefficient}
%\subsection{Jumping coefficient}
In this test, we consider a thick square $[0,2]\times [0,2]\times [0,1]$  composed by  9 patches, as represented in Figure \ref{fig:thick_square}. 
According to the procedure described at the beginning of Section \ref{sec:ptc_union},
we split $\Omega$ into 4 subdomains, each   obtained by merging the 4 patches around an inner edge. We  %set the right-hand side equal to $\underline{f}=[0, \ 0, \ -1]^T$ and we
 impose homogeneous Dirichlet boundary conditions everywhere, except for the side $\{z=1\}$, where we impose homogeneous Neumann boundary conditions. We consider jumping Lamé parameters, by setting $\nu=0.3$ and $E=1$ on green patches and $E=6$ on red patches. %The numerical solution is represented in Figure \ref{fig:thick_square_sol}.
The multilinear ranks of the blocks of $\widetilde{\mathbf{A}}$ are 
$$
 \left(R^A_1(i,i,k,k),R^A_2(i,i,k,k),R^A_3(i,i,k,k)\right) = (6,6,6), \qquad k=1,2,3,
$$ 
and
$$\left(R^A_1(i,i,k,\ell), R^A_2(i,i,k,\ell),R^A_3(i,i,k,\ell)\right) = (4,4,4), \qquad k,\ell = 1,2,3, \; k \neq \ell, $$
for $i = 1,\ldots,4$,
$$
\left(R^A_1(i,j,k,k), R^A_2(i,j,k,k),R^A_3(i,j,k,k)\right) = (3,3,3), \qquad k,\ell = 1,2,3, \; k \neq \ell\; i\neq j,
$$
and
$$
\left(R^A_1(i,j,k,\ell), R^A_2(i,j,k,\ell),R^A_3(i,j,k,\ell)\right) = (2,2,2), \qquad k,\ell = 1,2,3, \; k \neq \ell\; i\neq j,
$$
for $i,j = 1,\ldots,4$. 
}

\begin{figure}[H]
\begin{center}
\includegraphics[scale=0.35]{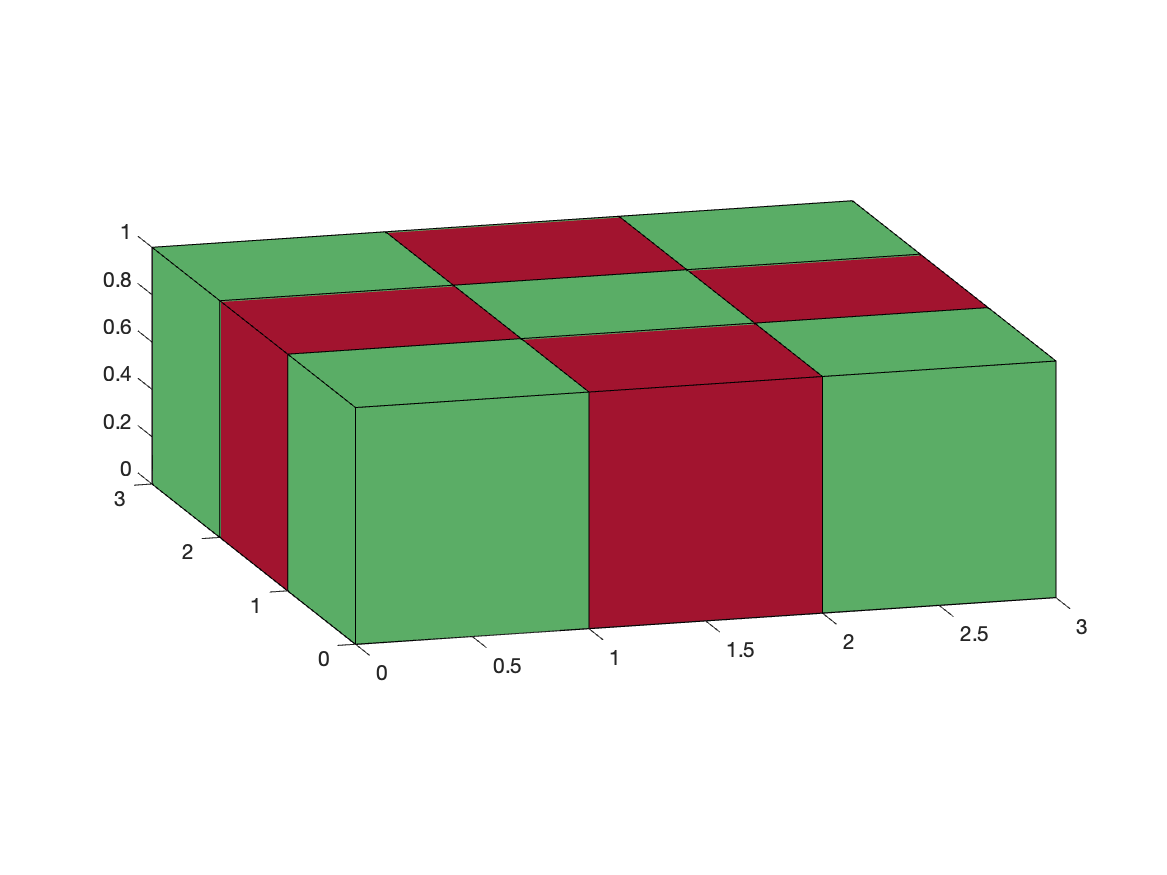}
  \caption{Thick square domain.}\label{fig:thick_square}
     \end{center}
\end{figure}

{\R1

In Table \ref{tab:its_thick_square} we report the numbers of iterations for degrees $p=3,4,5$ and number of elements per direction $n_{el}=2^L$ with $L=4,5,6,7$. The number of iterations is almost independent of degree and mesh-size.
 {\renewcommand\arraystretch{1.4} 
\begin{table}[H]
\begin{center}
%\footnotesize
\begin{tabular}{|c|c|c|c|}
\hline
 & \multicolumn{3}{|c|}{ \ Iteration number} \\
 \hline$n_{el}$ & $p=3$ &$p=4$  & $p=5$ \\
\hline   
16 &      68    &    65  &  66 \\ 
\hline
32 &     71    &  67   & 69  \\
\hline
64 &      71    &  70  &  72   \\ 
\hline
128 &       78     &   73 &   75  \\ 
\hline
\end{tabular}
\caption{Number of iterations for the thick square domain.}
\label{tab:its_thick_square}
\end{center}
\end{table}}

\R1
The maximum of the multilinear ranks of the solution in all the directions and for all the subdomains is shownn in Figure \ref{fig:rank_thick_square}. As the number of elements per direction $n_{el}$ grows, the multilinear rank of the solution increases, but, as represented in Figure \ref{fig:mem_thick_square}, the memory storage for the solution  hugely reduces with respect to the full solution.

\begin{figure}[H]
 \centering  
     \subfloat[][Maximum of the multilinear rank of the solution.\label{fig:rank_thick_square}] 
   {\includegraphics[scale=0.40]{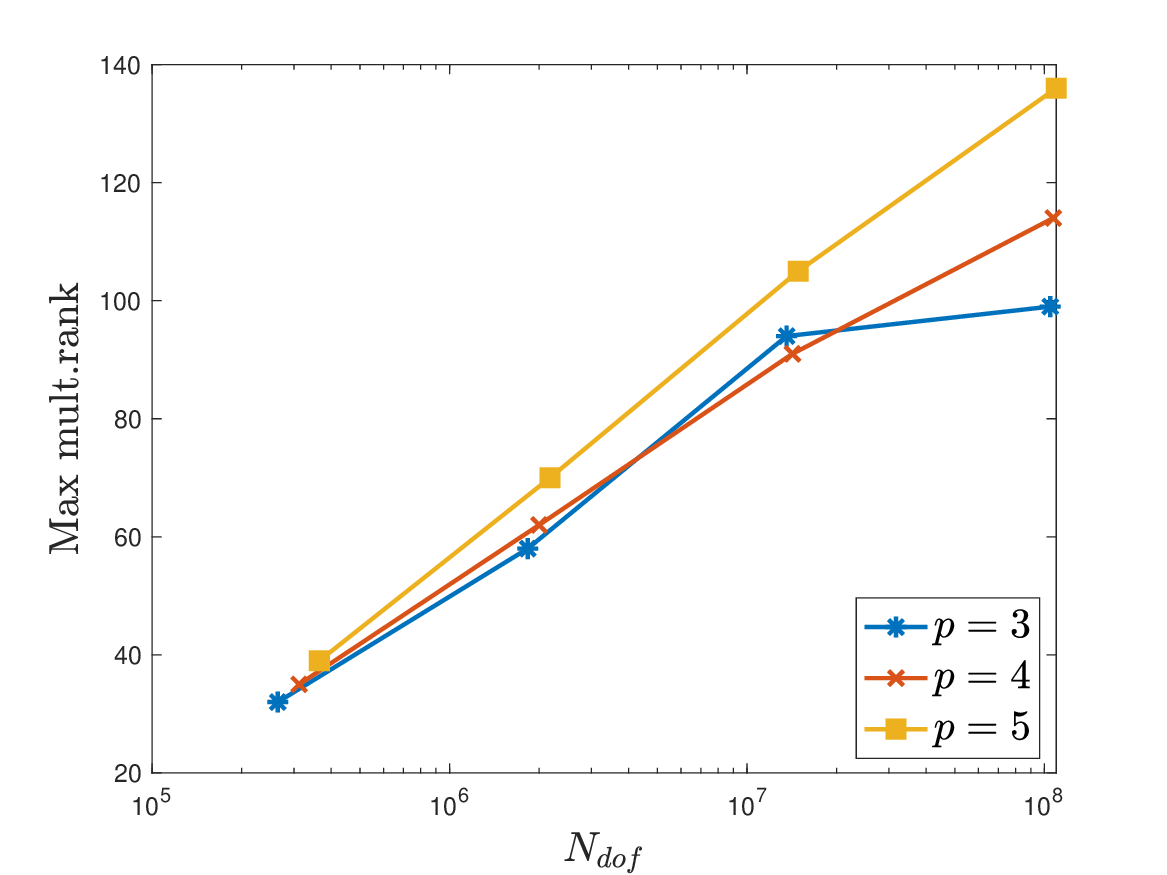}}\quad
 \subfloat[][Memory requirements.\label{fig:mem_thick_square}]
   {\includegraphics[scale=0.40]{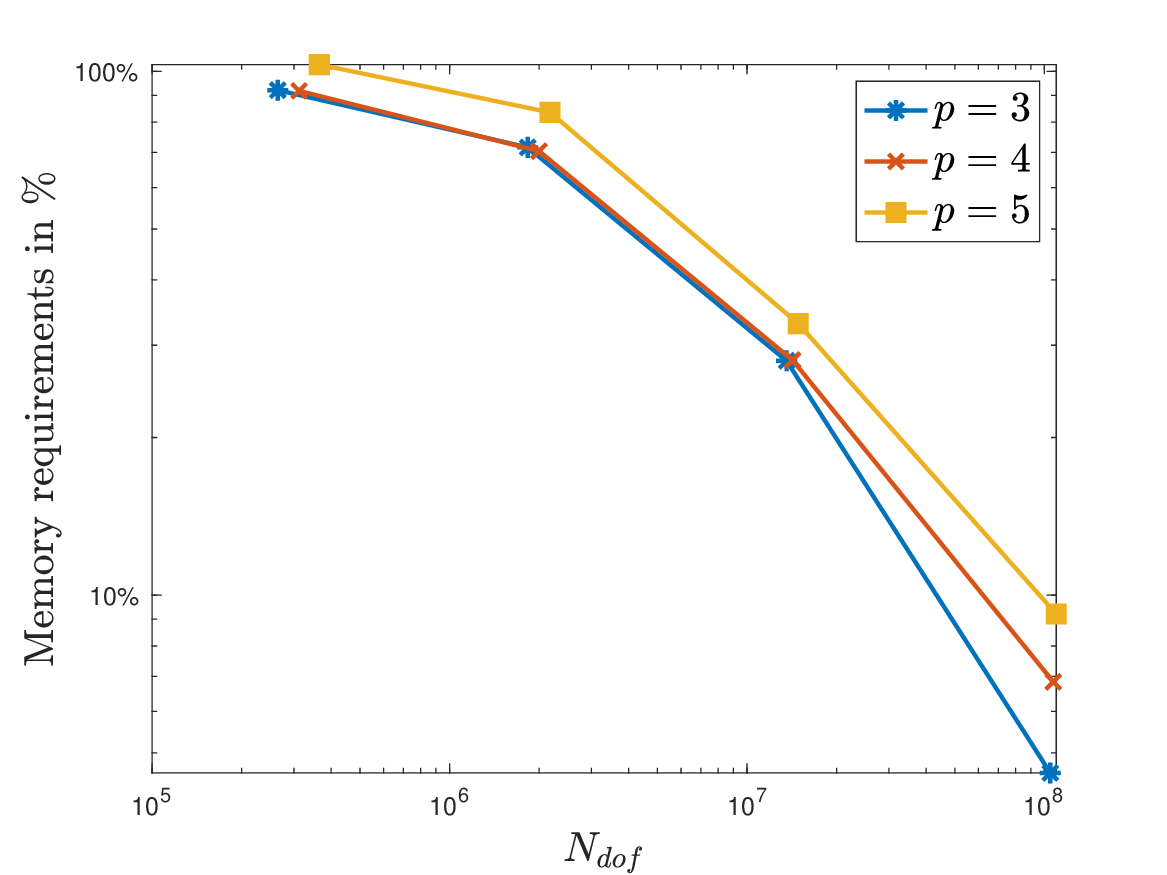}}
   \caption{Results for the thick square domain.}
 \label{fig:thick_square_results}
\end{figure}
}

\subsection{Thick ring domain}
In this test we consider a thick ring domain, represented in Figure \ref{fig:in_dom_ring}. 
We have four patches and four subdomains, represented by the union of green and yellow patches, yellow and blue patches, blue and red patches and red and green patches.

{\R1
For this domain, we preliminary perform a convergence test to assess that the numerical solution yielded by our approach is indeed a good approximation of the Galerkin solution. For this purpose, we consider a Poisson toy problem with homogeneous Dirichlet boundary conditions and source function $f(x,y,z)=\sin(4 \pi x) \sin(4 \pi y)  \sin( 4 \pi z)$. Note that the exact solution of this toy problem is $u(x,y,z) = -\left(48 \pi^2 \right)^{-1} f(x,y,z) $.
We solve the problem with a low-rank strategy analogous to that described previously (in particular we fix $10^{-10}$ as tolerance for TPCG) and then compute the $H^1$ and $L^2$ norms of the error for the obtained solution, for different discretization levels and polynomial degrees. Results are shown in Figure \ref{fig:Thick_results_error} and confirm the validity of our approach.

\begin{figure}[H]
 \centering  
     \subfloat[][$L^2$ error.\label{fig:L2_thick}] 
   {\includegraphics[scale=0.43]{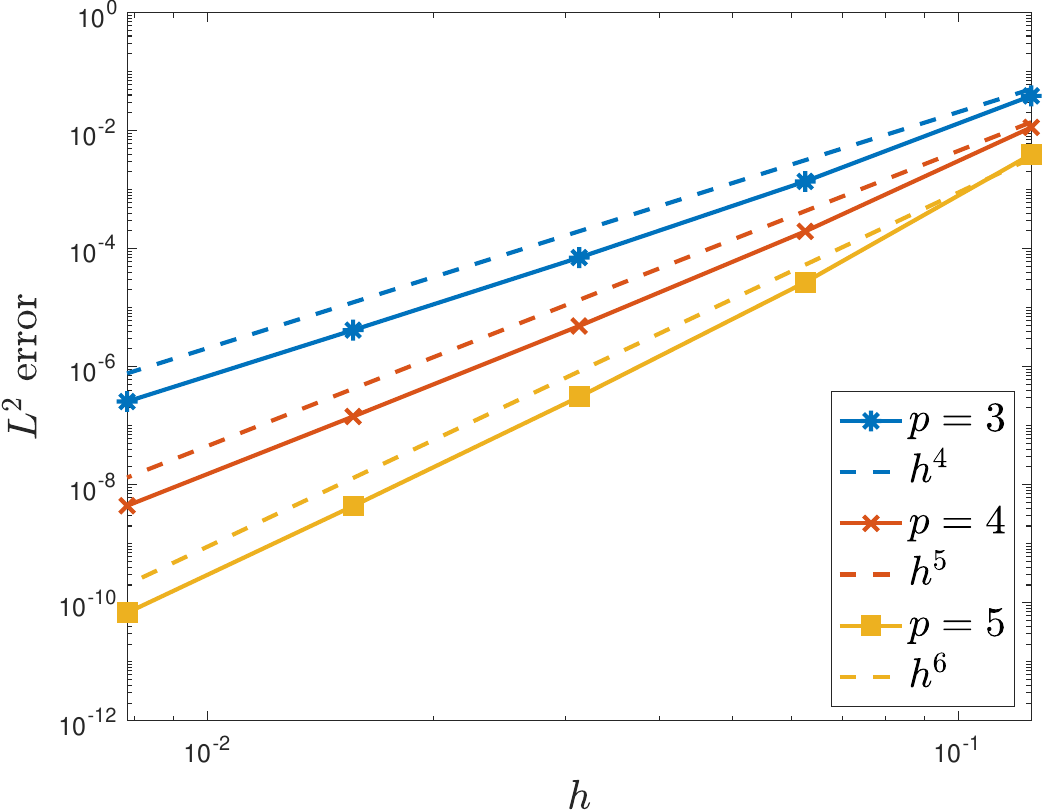}}\quad
 \subfloat[][$H^1$ error.\label{fig:H1_thick}]
   {\includegraphics[scale=0.43]{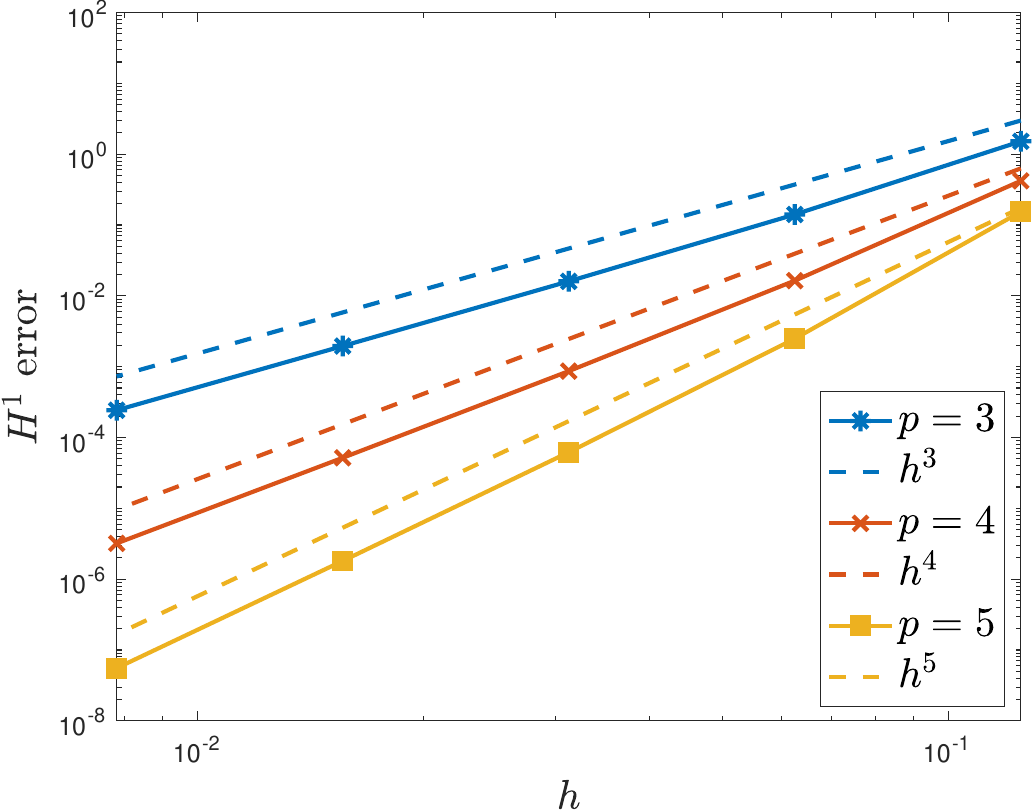}}
   \caption{Errors for a Poisson problem in the  thick ring domain.}
 \label{fig:Thick_results_error}
\end{figure}}

We  now solve a linear elasticity problem on the same domain. We
impose homogeneous Dirichlet boundary conditions on $\{z=0\}$ and in the central walls of the thick ring, while homogeneous Neumann boundary conditions are imposed elsewhere.% We the set the right-hand side equal to $\underline{f}=[0, \ 0, \ -0.5]^T$. 

\begin{figure}[H]
\centering \includegraphics[scale=0.40]{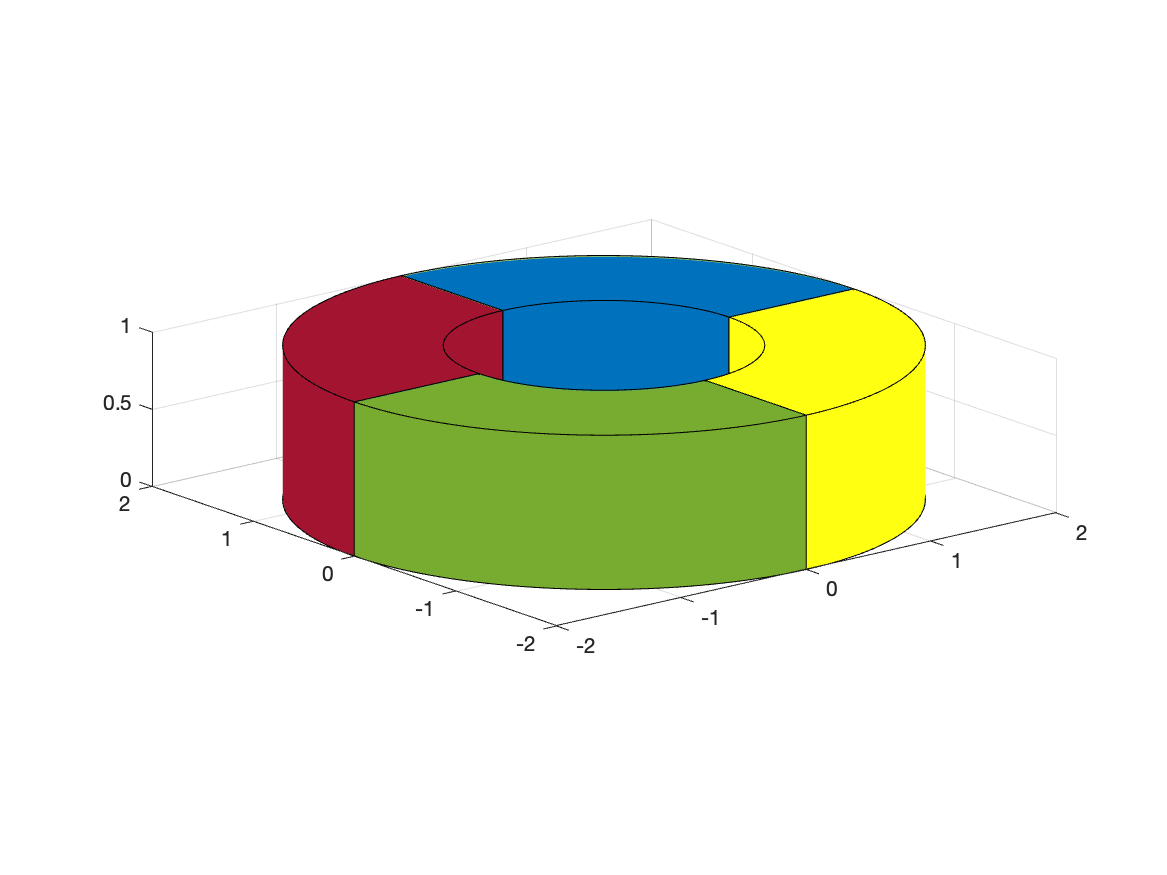}
  \caption{Thick ring domain.}\label{fig:in_dom_ring}
\end{figure}

For this domain, the multilinear ranks of the blocks of $\widetilde{\matb{A}}$ %(note that here the geometry is trivial for all patches, therefore $\matb{A} = \widetilde{\matb{A}}$) 
are
$$ \left(R^A_1(i,j,1,1), R^A_2(i,j,1,1),R^A_3(i,j,1,1)\right) =  \left(R^A_1(i,j,2,2), R^A_2(i,j,1,1),R^A_3(i,j,2,2)\right) = (5,5,5), $$
  $$ \left(R^A_1(i,j,3,3), R^A_2(i,j,3,3),R^A_3(i,j,3,3)\right) = (3,3,3) $$
and 
$$ \left( R^A_1(i,j,k,\ell), R^A_2(i,j,k,\ell),R^A_3(i,j,k,\ell)\right) =  (4,4,4), \qquad k,\ell = 1,2,3, \; k \neq \ell,$$
for $i,j = 1,2,3,4$ s.t. $\Theta^{(i)}\cap \Theta^{(j)}$ has non-zero measure.

{\R1 
The relative errors in 2-norm $\| \cdot \|_2$ of the approximation of the linear system matrix $\mathbf{A}$ by $\widetilde{\mathbf{A}}$ for  $p=3,4,5$ and  $n_{el}=8, 16$ are reported in Table \ref{tab:rel_err_ring}.
These computations require the assembly of both matrices, which forced us to consider  a limited number of elements per parametric direction, because of memory constraints. The resulting error is of the order of the {\tt Chebfun3F} tolerance, which we recall is set equal to $10^{-7}$ for all the considered cases.
 
 {  \renewcommand\arraystretch{1.4} 
\begin{table}[H]
\begin{center}
%\footnotesize
\begin{tabular}{|c|c|c|c|}
\hline
 & \multicolumn{3}{|c|}{ \ $\| \mathbf{A}-\widetilde{\mathbf{A}} \|_2 / \|\mathbf{A} \|_2$} \\
 \hline
 $n_{el}$ & $p=3$ &$p=4$  & $p=5$ \\  
 \hline
8   &  $ 4.9 \cdot 10^{-8}$  &  $ 3.6\cdot 10^{-8}$ & $  3.2 \cdot 10^{-8}$ \\
\hline
16  &  $  9.1 \cdot 10^{-8}$   &  $  8.3 \cdot 10^{-8}$ &  $ 7.6\cdot 10^{-8}$\\   
\hline
  
\end{tabular}
\caption{Relative errors of the approximation of $\mathbf{A}$ by $\widetilde{\mathbf{A}}$ in the thick ring domain.}
\label{tab:rel_err_ring}
\end{center}
\end{table}}
    }

Table \ref{tab:its_ring} collects the number of iterations for degrees
$p= 3,4,5$ and number of elements per patch per parametric direction
$n_{el}=2^{L}$ with $L=5,6,7,8$, as in the previous test.   The number
of iterations is independent of the meshsize $h$  and appears to be
stable with respect to the polynomial degree  $p$.%, although the case $p=2$ shows higher iteration counts compared to the higher degrees.  

 {\renewcommand\arraystretch{1.4} 
\begin{table}[H]
\begin{center}
%\footnotesize
\begin{tabular}{|c|c|c|c|}
\hline
 & \multicolumn{3}{|c|}{ \ Iteration number} \\
 \hline$n_{el}$ &  $p=3$ &$p=4$  & $p=5$ \\
\hline  
32 &       42  &   39  &  40  \\
\hline
64 &      44   &  41   &    42  \\ 
\hline
128 &     50  &  46   &  46   \\ 
\hline
256 &      60   &   58  &  55   \\ 
\hline
\end{tabular}
\caption{Number of iterations for the thick ring domain.}
\label{tab:its_ring}
\end{center}
\end{table}}

We report in Figure \ref{fig:rank_ring} the maximum of the multilinear ranks of the solution in all the directions and for all the subdomains. %As in the previous test case, the solution has multilinear ranks that grow when $n_{el}$ increases. 
%Interestingly, for degree  4 this maximum rank appears to grow in the left part of of the plot (corresponding to small values of $n_{el}$), but it then decreases as $n_{el}$ grows and on the finest discretization level reaching a value between 60 and 70, which is also (roughly) the maximum rank observed for degree 3. 
As represented in Figure \ref{fig:mem_ring}, the memory storage for the solution is hugely reduced with respect to the full solution and reaches values around $1\%$ for the highest discretization level.

\begin{figure}[H]
 \centering  
     \subfloat[][Maximum of the multilinear rank of the solution.\label{fig:rank_ring}] 
   {\includegraphics[scale=0.43]{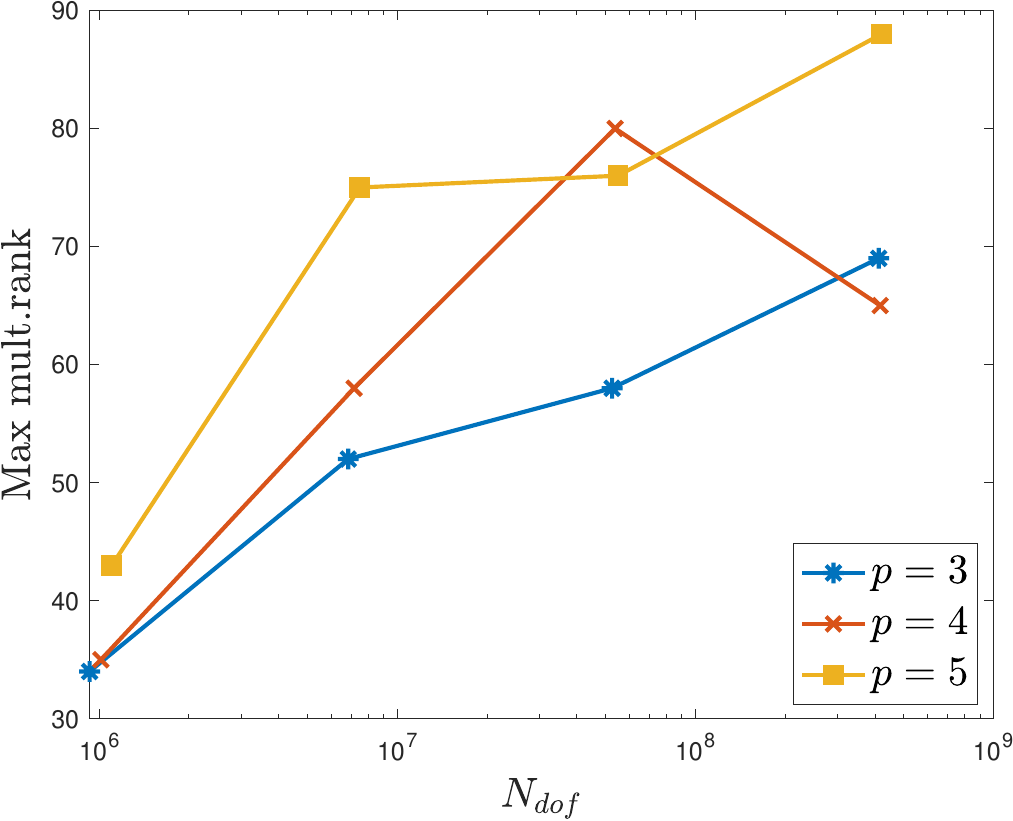}}\quad
 \subfloat[][Memory requirements.\label{fig:mem_ring}]
   {\includegraphics[scale=0.43]{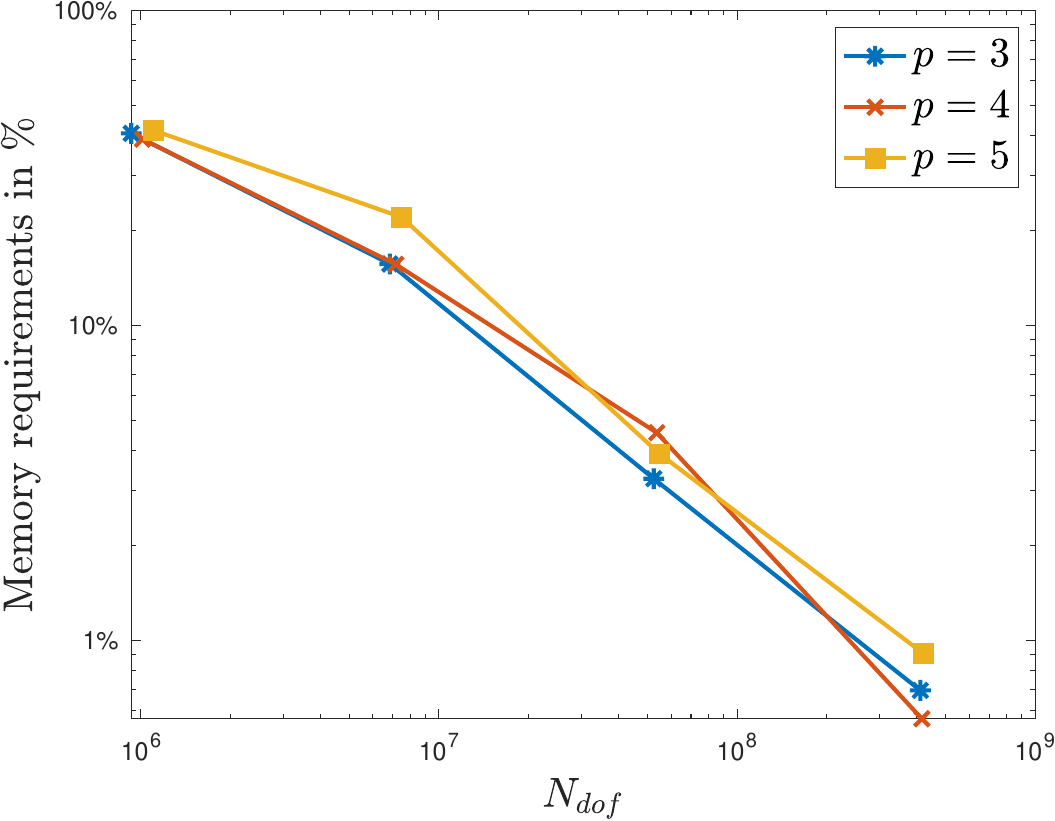}}
   \caption{Results for the thick ring domain.}
 \label{fig:ring_results}
\end{figure}

{
\subsection{Igloo-shaped domain}
In this test we consider an igloo-shaped domain, represented in Figure \ref{fig:igloo_domain}. As in the previous test case, we have four patches and four subdomains, represented by the union of green and yellow patches, yellow and blue patches, blue and red patches and red and green patches. We impose homogeneous Dirichlet boundary conditions on $\{z=0\}$ and the internal boundary, while homogeneous Neumann boundary conditions are imposed elsewhere.   %The numerical solution is represented in Figure \ref{fig:num_sol_igloo}. }

\begin{figure}[H]
\centering 
   \includegraphics[scale=0.40]{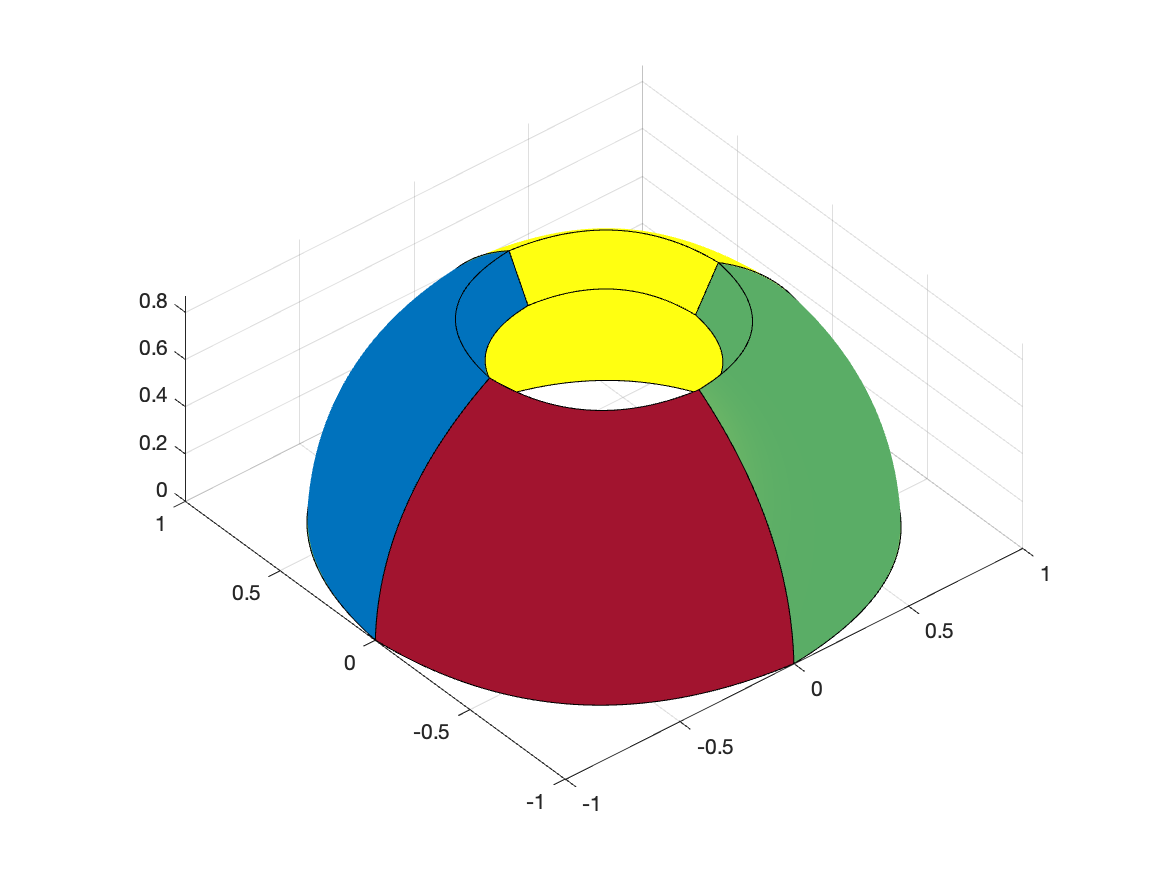} 
  \caption{Igloo-shaped domain.}\label{fig:igloo_domain}
\end{figure}

The multilinear ranks of the blocks of $\widetilde{\matb{A}}$ %(note that here the geometry is trivial for all patches, therefore $\matb{A} = \widetilde{\matb{A}}$) 
are
$$ \left(R^A_1(i,j,1,1), R^A_2(i,j,1,1),R^A_3(i,j,1,1)\right) =  \left(R^A_1(i,j,2,2), R^A_2(i,j,1,1),R^A_3(i,j,2,2)\right) = (13, 13, 9),$$
for $i,j=1,\dots,4$  s.t. $\Theta^{(i)}\cap \Theta^{(j)}$ has zero measure, while
  $$ \left(R^A_1(i,i,3,3), R^A_2(i,i,3,3),R^A_3(i,i,3,3)\right) = (13, 13, 9) \quad \text{for } i=1,\dots,4$$
   $$ \left(R^A_1(i,j,3,3), R^A_2(i,j,3,3),R^A_3(i,j,3,3)\right) = (12, 12, 9) \quad \text{for } i\neq j\text{ and } i,j=1,\dots,4,$$
   s.t. $\Theta^{(i)}\cap \Theta^{(j)}$ has non-zero measure. Finally,
$$ \left( R^A_1(i,j,k,l), R^A_2(i,j,k,l),R^A_3(i,j,k,l)\right) = (9,9,9), \qquad \text{if } (k,l)\in\{(1,2),(2,1)\},  $$
and
$$ \left( R^A_1(i,j,k,l), R^A_2(i,j,k,l),R^A_3(i,j,k,l)\right) = (8,8,8), \qquad \text{if } (k,l)\in\{(1,3),(3,1),(2,3),(3,2)\}  $$
 for $i,j=1,\dots,4$  s.t. $\Theta^{(i)}\cap \Theta^{(j)}$ has zero measure.
 } Also in this testcase, we report in Table \ref{tab:rel_err_igloo}  the relative errors in 2-norm $\| \cdot \|_2$ of the approximation of the linear system matrix $\mathbf{A}$ by $\widetilde{\mathbf{A}}$ for  $p=3,4$ and  $n_{el}=8, 16$. %Note that we considered a limited number of elements per parametric direction, because of memory constraints, as these computations require the assembly of both the matrices. As expected, the approximation is of the order of the Chebfun tolerance, which we recall is set equal to $10^{-7}$ for all the cases considered.
Even in this case, the approximation is of the order of the Chebfun tolerance, which is set equal to $10^{-7}$ for all the cases considered.

 {  \renewcommand\arraystretch{1.4} 
\begin{table}[H]
\begin{center}
%\footnotesize
\begin{tabular}{|c|c|c|}
\hline
 & \multicolumn{2}{|c|}{ \ $\| \mathbf{A}-\widetilde{\mathbf{A}} \|_2 / \|\mathbf{A} \|_2$} \\
 \hline
 $n_{el}$ & $p=3$ &$p=4$    \\  
 \hline
8   &  $3.7\cdot 10^{-8}$  &  $3.3\cdot 10^{-8}$  \\
\hline
16  &  $7.0\cdot 10^{-8}$   & $ 6.7 \cdot 10^{-8}$   \\   
\hline
  
\end{tabular}
\caption{Relative errors of the approximation of $\mathbf{A}$ by $\widetilde{\mathbf{A}}$ in the igloo-shaped domain.}
\label{tab:rel_err_igloo}
\end{center}
\end{table}}

%We remark that  this geometry has higher multilinear ranks than the previous ones. Consequently, numerical experiments are run up to a lower number of elements per parametric direction and degrees. 
In Table \ref{tab:its_igloo} we reported the number of iterations for degrees $p = 3,4$ and a number of elements per parametric direction $n_{el}=2^L$ with $L=3,4,5,6$.
Note that the maximum values of $L$ and $p$ are  slightly smaller than in the two previous tests. This is due to the higher multilinear ranks of this geometry. Nevertheless, note that the number of degrees of freedom for the finest discretizations level is higher than 7 millions.

 {  \renewcommand\arraystretch{1.4} 
\begin{table}[H]
\begin{center}
%\footnotesize 
\begin{tabular}{|c|c|c|}
\hline
 & \multicolumn{2}{|c|}{ Iteration number} \\
 \hline 
 $n_{el}$ &   $p=3$ &  $p=4$    \\
 \hline
8  & \quad 47 \ \ \  & 43  \\
\hline  
16 &     \quad    49 \ \ \   &  45   \\  
\hline
32 &    \quad   52  \ \ \  &   51  \\
\hline
64 &    \quad   59 \ \ \   &  58    \\  
\hline
\end{tabular}
\caption{Number of iterations for the igloo-shaped  domain.}
\label{tab:its_igloo}
\end{center}
\end{table}}
{
The maximum of the multilinear ranks of the solution in all the directions and for all the subdomains are represented in Figure \ref{fig:rank_igloo}. 
The gain in terms of storage may seem less significant than in the previous cases, but it is in fact very similar if the comparison is done for the same discretization level. For this problem, on the finest discretization level, the memory storage of the low-rank solution is less than 35\% that of the full solution.

%The multilinear ranks of the solution increase, as $n_{el}$ increases, as in the previous test cases.
% However, the gain in terms of storage cost is evident, as int he previous test cases: the memory storage of the solution is much less with respect to the full solution  and reaches values below 35\% for the highest discretization level.
}

\begin{figure}[H]
 \centering  
     \subfloat[][Maximum of the multilinear rank of the solution.\label{fig:rank_igloo}] 
   {\includegraphics[scale=0.4]{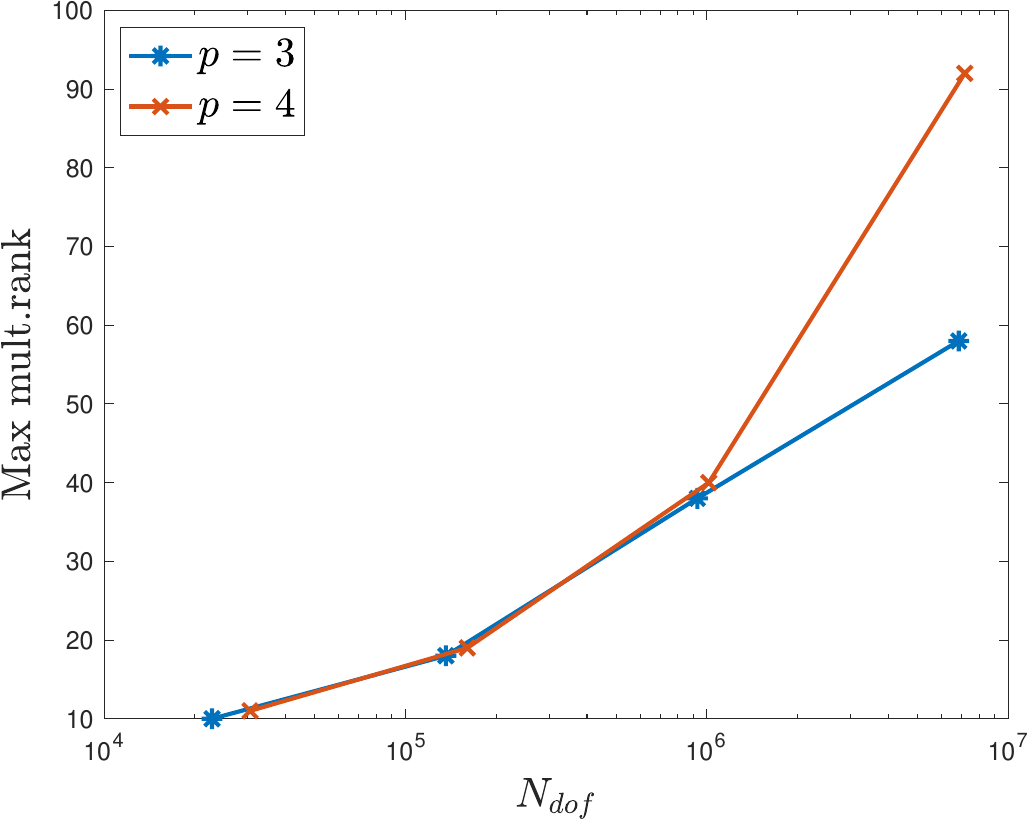}}\quad
 \subfloat[][Memory requirements.\label{fig:mem_igloo}]
   {\includegraphics[scale=0.4]{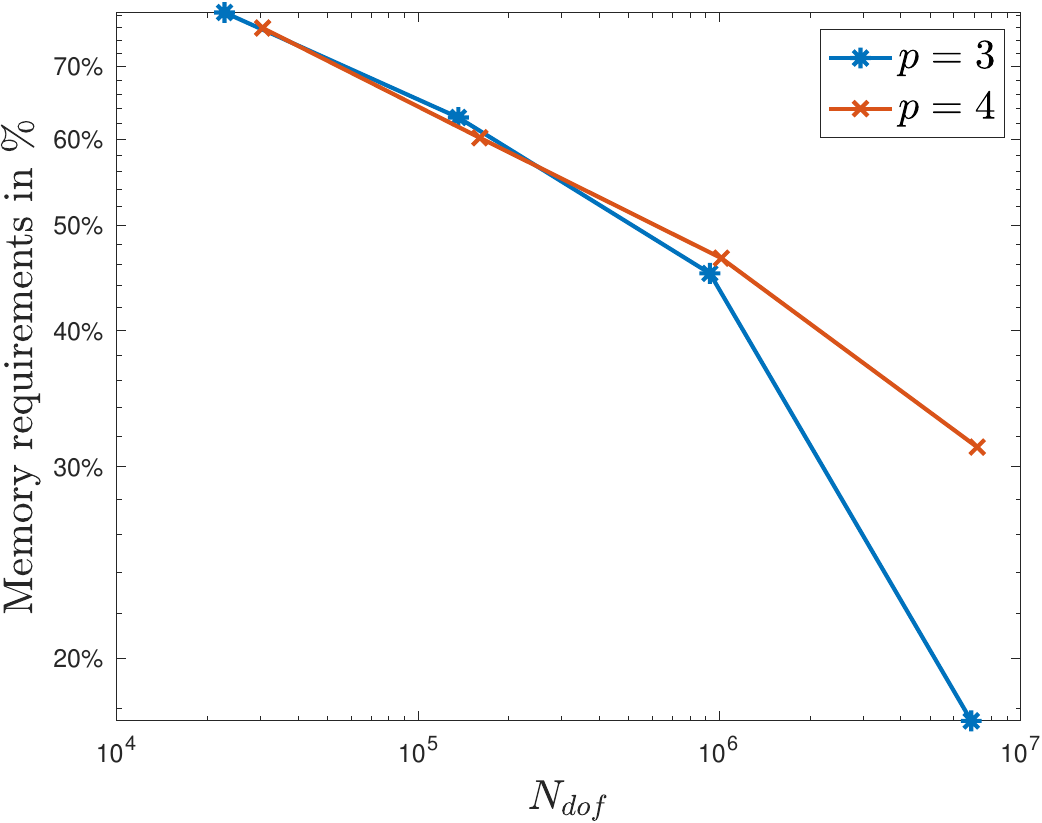}}
   \caption{Results for the igloo-shaped domain.}
 \label{fig:igloo_results}
\end{figure}

\section{Conclusions}
\label{sec:conclusions}

In this work, we have considered the isogeometric discretization of a compressible linear elasticity problem. 
We have extended the low-rank solution method presented \cite{Montardini2023}, which is based on a low-rank Tucker representation of vectors and matrices, to conforming multipatch discretizations.
This solver exploits the domain decomposition paradigm,  where the subdomains are built from union of adjoining patches.
%
% In ??this work we have extended the method presented an isogeometric  method for compressible linear elasticity problems, based on a low-rank Tucker representation of vectors and matrices. 
%
% In this work we have presented an isogeometric multipatch method in low-rank format for compressible linear elasticity problems, based on a Tucker representation of vectors and matrices. 
% The multipatch method considers overlapping subdomains, that are built from union of two conforming adjoining patches. 
Each block of the system matrix and each block of the right hand side
vector is approximated by a Tucker matrix and a Tucker vector,
respectively. 
In particular, we showed an upper bound for the relative error of the matrix stemming from this approximation.
The resulting singular linear system is solved by a truncated preconditioned conjugate 
 gradient method. The designed preconditioner has a block-diagonal
 structure where each block is a Tucker matrix and its application
 exploits the fast diagonalization method and the fast Fourier
 transform. 
 %We also showed a theoretical upper bound on the relative error of the system matrix. 

We performed numerical tests assessing a low memory storage, almost
 independent of the degree $p$, and the number of  iterations
 independent of $p$ and the mesh-size. 
 The problems considered are simple academic
 benchmarks, and we emphasize  that the effectiveness of
 these techniques is greater when the rank of the considered
 approximation of the problem's components (the geometry, the
 operator coefficients, the forcing term, and therefore the solution)
 is lower. The promising results obtained suggest further exploring,  in
 future work, the potential of the proposed method in realistic
 applications. The method  relies on the linearity
 of the problem, however, nonlinear problems could be addressed by
 combining outer Newton-type solvers with inner solvers of the
 proposed type for the linearized step.

Our tests also  demonstrate that our procedure achieves the optimal
accuracy expected by the isogeometric scheme. 
However, this paper does
not delve into the theoretical error study. {\GS To demonstrate an error
estimate for the solution for the error stemming from the use of a
low-rank approximation of the data and the differential operator of
the problem, as well as from the low-rank approximation of the
variables constructed by the iterative solver is of primary importance
to provide the mathematical justification for the use of low-rank
techniques applied to solutions of partial differential equations in
general, not just in the multipatch case. A complete  error theory
of this kind is not yet available and represents a central aspect of
the theoretical development of these techniques.
 The topic warrants further investigation and will be the focus of future researches.
}

\section*{Acknowledgments}

The authors are members of the Gruppo Nazionale Calcolo
Scientifico-Istituto Nazionale di Alta Matematica (GNCS-INDAM), { the
first author was partially supported by INDAM-GNCS Project  
``Sviluppo di metodi numerici innovativi ed efficienti per la risoluzione di PDE'', while the third author was partially supported by INDAM-GNCS Project  ``Nodi, ascisse e punti: scegliere e usare in maniera efficiente''.   The second author  acknowledges support from
PNRR-M4C2-I1.4-NC-HPC-Spoke6.  The third author  acknowledges the contribution of the Italian Ministry of
University and Research (MUR) through the PRIN project COSMIC
(No. 2022A79M75).  The first author acknowledges support from the PRINN 2022 PNRR project HEXAGON (No. P20227CTY3). The second and the third author acknowledge support from    the PRIN 2022 PNRR project NOTES
(No. P2022NC97R)}.

\bibliographystyle{plain}
 \bibliography{biblio_tensor_solver}

\end{document}